\documentclass{amsart}
\usepackage[a4paper, margin=1in]{geometry}
\usepackage[utf8]{inputenc}
\usepackage{hyperref}
\usepackage{enumitem}
\usepackage{xcolor}
\usepackage{ulem}
\usepackage{amsthm}
\usepackage{amsmath}
\usepackage{mathrsfs}
\usepackage[numbers]{natbib}
\usepackage{comment}

\usepackage{
    mathtools,
    amsfonts,
    amssymb,
    amsthm,
    nicefrac,
    empheq
}

\newtheorem{theorem}{Theorem}
\newtheorem{definition}{Definition}
\newtheorem{prop}{Proposition}
\newtheorem{remark}{Remark}

\newtheorem{lemma}[theorem]{Lemma}
\newtheorem{ass}{Assumption}

\newtheorem*{notation}{Notation}

\renewenvironment{proof}[1][Proof]{%
    \par\noindent{\bfseries #1. }}{\qed}

\newlist{condlist}{enumerate}{1}
\setlist[condlist, 1]{label=(\Alph*)}

\newcommand{\R}{\mathbb{R}}
\newcommand{\N}{\mathbb{N}}
\newcommand{\Z}{\mathbb{Z}}
\newcommand{\T}{\mathbb{T}}
\newcommand{\HH}{\mathbb{H}}
\newcommand{\E}{\mathbb{E}}

\newcommand{\LL}{\mathbb{L}}
\newcommand{\prob}{\mathbb{P}}
\newcommand{\curl}{\text{curl}}
\newcommand{\W}{\mathbb{W}}

\newcommand{\norm}[1]{\left\lVert #1 \right\rVert}

\newcommand{\indicator}{\mathbb{I}}

\title[Uniform approximation 2d-SNS]{A uniform point vortex approximation for the solution of the two dimensional Navier Stokes equation with transport noise}
\author[F. Giovagnini]{Filippo Giovagnini$^{\star}$}
\author[D. Crisan]{Dan Crisan$^{\dagger}$}
\thanks{$^\star$Department of Mathematics, Imperial College London. Corresponding author's e-mail: \href{mailto:f.giovagnini23@imperial.ac.uk}{f.giovagnini23@imperial.ac.uk}}
\thanks{$^\dagger$Department of Mathematics, Imperial College London.}

\begin{document}

\maketitle

\begin{abstract}
We study a model of interacting particles represented by a system of $N$ stochastic differential equations. We establish that the (mollified) empirical distribution of the system converges uniformly with respect to both time and spatial variables to the solution of the two-dimensional Navier–Stokes equation with transport noise. The proofs are based on a semigroup approach.
\end{abstract}

%%%%%%%%%%%%%%%%%%%%%%%%%%%%%%%%%%%%%%%%%%%%%%%%%%%%%
\tableofcontents

\section{Introduction}

This paper investigates the mean-field limit of an interacting particle system and demonstrates that, after mollification, the limit is governed by the two-dimensional stochastic Navier–Stokes equations (2d-SNS). The particle system was originally introduced by Chorin \cite{Chorin1994VorticityAT}, with the first rigorous convergence result established by Marchioro and Pulvirenti \cite{Marchioro1982HydrodynamicsIT}. Subsequent work by Long \cite{Long1988} obtained a convergence rate, while Méléard \cite{meleard,meleard_monte_carlo} strengthened these results by proving convergence in the path space of the empirical measures.
A key refinement was introduced by Flandoli, Olivera, and Simon \cite{paper}, who considered a \emph{mollified} version of the system. Their approach improved the convergence, upgrading the distributional result of \cite{Marchioro1982HydrodynamicsIT} to a uniform one. The purpose of the present work is to extend this stronger uniform convergence to a stochastic setting, where each particle is additionally influenced by a common multiplicative noise.
More precisely, we prove that the mollified empirical measure of the particle system converges uniformly in time and space to the unique solution of
\begin{equation}
\label{eq:main_equation}
\begin{cases}
& \partial_t\omega(t, x) + u(t, x)\cdot\nabla\omega(t, x),dt + \sum_{k\in K} \sigma_{k}(x) \circ dW^k_t\cdot\nabla\omega(t, x) = \nu\Delta\omega(t, x),dt, \quad x \in \T^2,\\
& u(t, x) = K*\omega(t, x), \\
& \omega_{|t=0}(x) = \omega_0(x),
\end{cases}
\end{equation}
where $\omega_0:\T^2 \to \R^2$ is such that $\int_{\T^2}\omega_0(x)dx = 0$.
Here, $\circ dW^k_t$ denotes a Stratonovich-type stochastic integral, and $K$ is the Biot–Savart kernel (see the Appendix for details).
This work connects with a broader literature on stochastic mean-field limits. In particular, Guo and Luo \cite{GuoLuo2023} also study a particle system under a common noise, but in their setting the noise is scaled in such a way that it vanishes in the limit, leading to a deterministic partial differential equation (PDE) rather than a stochastic one (SPDE). Moreover, their convergence result is not formulated in fractional Sobolev spaces, in contrast with the approach we pursue here. See also \cite{Coghi2016Propagation}, \cite{Flandoli2024MeanField}, \cite{doi:10.1142/S021949372040002X}, and \cite{rosenzweig2020meanfieldlimitstochasticpoint} for related perspectives.
The paper is organized as follows. Section \ref{sec:notation} introduces notation and recalls preliminary results. Section \ref{sec:particle-system} presents the particle system and states the main theorem. Section \ref{sec:compactness} establishes compactness of the associated laws. Section \ref{sec:passing-to-the-limit} proves that any accumulation point of the sequence satisfies the 2d-SNS equation. Section \ref{sec:lifting} addresses the lifting of the truncation. Finally, Section \ref{sec:convergence-probability} proves convergence in probability of the full sequence, eliminating the need for subsequence arguments.

\newpage
%%%%%%%%%%%%%%%%%%%%%%%%%%%%%%%%%%%%%%%%%%%%%%%%%%%%%
\section{Notation}
\label{sec:notation}

Here is the notation we will use through the paper.

\begin{notation}
$\left.\right.$\\[-5mm]
\begin{itemize}
    \item We write \( a \lesssim b \) when there exists a positive constant \( C \), independent of \( a \) and \( b \), such that $a \leq C b$.
    \item If $f, g:\T^2 \to \R$, we define $f \star g : \T^2 \to \R$ as follows:
    \[
    f \star g (x) := \int_{\T^3} f(x-y) g(y) dy, \quad x \in \T^2.
    \]
    \item If $p \geq 1$, $\LL^p(\T^2; \R)$ denotes the usual $\LL^p$ space on the torus.
    \item For $p \geq 1$ and $\alpha \in \R$:
    \[
    \HH^{\alpha}_p(\T^2):= \left\{ f \in \LL^p(\T^2) \bigg| \sum\limits_{k \in \Z^3 } \left(1 + |k|^2 \right)^{\alpha/2} c_k(f) \textbf{e}_k(x) \in \LL^p(\T^2)\right\},
    \]
    where $\textbf{e}_k: \T^2 \to \mathbb{C}$ is defined as:
    \[
    \textbf{e}_k(x) := \frac{1}{(2 \cdot \pi)^{\frac{3}{2}}} e^{i <k, x>}.
    \]
    and where $c_k(f)$ is the $k$-th Fourier coefficient of $f$. For a connection with the usual definition with the classical Sobolev spaces, see \cite{hitchhiker}.
    \item For $f:\T^2 \to \R$, we will write $\mathcal{F}(f)$ for the Fourier transform of $f$.
    \item Let $\chi$ be supported in $\{x: |x|<c \}$ and $\rho$ supported in $\{x: a \leq |x| \leq b\}$ such that:
    \begin{enumerate}
        \item $\chi+\sum_{j \geq 0} \rho(2^{-j}) \equiv 1$,
        \item $\text{supp}(\chi) \cap \text{supp}(\rho(2^{-j}\cdot )) = \emptyset$ for $j \geq 1$ and $\text{supp}(\rho(2^{-i}\cdot )) \cap \text{supp}(\rho(2^{-j}\cdot )) = \emptyset$ for all $i, j \geq 0$ with $|i-j| \geq 1$.
    \end{enumerate}
    Define $\rho_{-1}$ and $\rho_j=\rho(2^{-j}\cdot)$. Now define $\Delta_j f=\mathcal{F}^{-1}(\rho_j \mathcal{F}(f))$. Given $p, q \in [1, +\infty]$ we define the Besov spaces as:
    \[
    B^{\alpha}_{p, q} = \left\{f: \norm{f}_{B^{\alpha}_{p, q}}:= \left( \sum\limits_{j \geq -1} \left( 2^{j\alpha}\norm{\Delta_j f}_{\LL^p(\T^2)}\right)^q\right)^{1/q} < +\infty \right\}
    \]
    \item For a $C^2$ function $f:\T^2 \to \R$, $\nabla^2 f: \T^2 \to \R^{2 \times 2}$ denotes the Hessian matrix of $f$.
    \item $\Delta : \HH^{2}_p(\T^2) \subset \LL^p(\T^2) \to \LL^p(\T^2)$ is the usual Laplacian.
    \item $\indicator : \LL^p(\T^2) \to \LL^p(\T^2)$ is the identity operator.
    \item $(\indicator - \alpha^2 \Delta)^{-1} :  \LL^p(\T^2) \to \HH^{2}_p(\T^2)$ is the inverse of the differential operator $\indicator - \alpha^2 \Delta$.
    \item For $-\alpha > 0$ the following characterization holds:
    \[
    \HH^{-\alpha}_p(\T^2) = \left( \HH^{\alpha}_p(\T^2) \right)^*
    \]
    where $\left( \cdot \right)^*$ denotes the topological dual of a Banach space.
    \item $C([0, T]; \HH^{\alpha}_p(\T^2))$ is the space of functions $f:[0, T] \to \HH^{\alpha}_p(\T^2) $ that are continuous on $[0,T]$.
    \item For two functions $f,g:D \to \R$ we will denote:
    \[
    \left\langle f,g \right\rangle:= \int_D f(x) \cdot g(x) dx,
    \]
    where $D=\T^2$ or $D=\R^2$.
    \item $\left(\Xi, (\mathcal{F}_t)_{t \geq 0}, \prob \right)$ will denote a filtrated probability space and  $\left(\Xi, (\mathcal{F}_t)_{t \geq 0}, \prob, (W_t)_{t \geq 0} \right)$ will denote a $\R$ valued Wiener process, unless otherwise stated.
    \item Given a vector field $\sigma : \T^2 \to \R^2$, we write $\sigma \cdot \nabla \sigma$ matrix-vector multiplication $\left(\nabla \sigma \right) \sigma $.
    \item For a probability measure $\mu \in \mathcal{P}(\T^2)$ and for a function $f: \T^2 \to \R$, we denote $f \star \mu(x)$ for $\int_{\T^2} f(x - y) \mu(dy)$, i.e. the integral of $f$ with respect to $\mu$ when it is well defined.
\end{itemize}
\end{notation}

%%%%%%%%%%%%%%%%%%%%%%%%%%%%%%%%%%%%%%%%%%%%%%%%%%%%%

%%%%%%%%%%%%%%%%%%%%%%%%%%%%%%%%%%%%%%%%%%%%%%%%%%%%%

%%%%%%%%%%%%%%%%%%%%%%%%%%%%%%%%%%%%%%%%%%%%%%%%%%%%%
\section{The particle system and the main result}
\label{sec:particle-system}
Let $\omega_0:\T^2 \to \R$ be the initial condition of the system \eqref{eq:main_equation}. From now on we will assume $\int_{\T^2} \omega_0(x)dx=0$, and we define $\Gamma^+:=\int_{\T^2} \max(\omega_0(x), 0) dx$ and $\Gamma^-:=\int_{\T^2} \max(-\omega_0(x), 0) dx$. This condition is necessary being $\omega$ the curl of a vector field on a compact set.

Consider:
\begin{equation}
\label{eq:particle_system}
\begin{cases}
dX_t^{i,N, +}:= & F \bigg( \left(V^N \star K\right) \star \mu^N_t (X^{i, N, +}_t) \bigg) dt + \sqrt{2 \nu} dB_t^{i, +} +\sum_{k \in \mathbb{N}} \sigma_k\left(X_t^{i,N, +}\right) \circ dW_k(t), \\
dX_t^{i,N, -}:= & F \bigg( \left(V^N \star K\right) \star \mu^N_t (X^{i, N, -}_t) \bigg) dt + \sqrt{2 \nu} dB_t^{i, -} +\sum\limits_{k \in \mathbb{N}} \sigma_k\left(X_t^{i,N, -}\right) \circ dW_k(t), \\
\mu^{N, \pm}_t =&  \frac{\Gamma^\pm}{N} \sum\limits_{j = 1}^N \delta_{X^{j, \pm}_t},\\
\mu^N_t = & \mu^{N, +}_t - \mu^{N, -}_t,\\
X_0^{i,N, \pm}= & X^{i, \pm}_0.\\
\end{cases}
\end{equation}

Here the function $F_M: \T^2 \to \R^2$ is defined as:
\begin{equation}
\label{eq:definition_of_the_truncation}
\begin{bmatrix}
           x_{1} \\
           x_{2}
         \end{bmatrix} \to \begin{bmatrix}
           (x_{1} \wedge M ) \vee (-M) \\
           (x_{2} \wedge M ) \vee (-M)
         \end{bmatrix} =: F \begin{bmatrix}
           x_{1} \\
           x_{2}
\end{bmatrix}.
\end{equation}

Since $M$ will always be a fixed positive constant, we will omit it for the sake of clarity and will just write $F=F_M$.

Furthermore, $K$ denotes the Biot-Savart kernel. For a detailed definition and for the proof of the main properties, see the Appendix.

The function $V^N$ is an approximation of the identity. For a rigorous definition and list of properties see the Appendix.

\begin{definition}
\label{def:solution_of_system}
Let $0<\eta<1$ and $p>2$. We will say that $(\omega^+, \omega^-)$ is a solution to \eqref{eq:main_equation} if: $\left( \Xi, \left( \mathcal{F}_t \right)_{t \geq 0}, \prob\right)$ is a filtered probability space, $W_k$ is a sequence of $\R$-valued independent Brownian Motions:
\[
(\omega^+, \omega^-): \left(\Xi \times[0, T] \times \mathbb{T}^2 \right)^2 \rightarrow \R^2,
\]
is an adapted process with $(\omega^+, \omega^-) \in C\left([0, T]; \HH_p^{\eta}\left(\mathbb{T}^2\right)\right)^2$ almost surely and for every $\phi, \psi \in C^{\infty}(\T^2)$ the following two hold:
\begin{equation}
\begin{aligned}
&\left\langle \omega^+_t ,\phi \right\rangle + \int_0^t \left\langle \omega^+_s, K \star \left( \omega^+_s - \omega^-_s \right) \cdot\nabla\phi \right\rangle ds =\nu \int_0^t \left\langle \omega^+_s, \Delta \phi \right\rangle ds + \\
& \int_0^t \sum_{k\in K}
\left\langle \omega^+_s, \sigma_{k} \cdot\nabla\phi \right\rangle dW^k_s + \frac{1}{2} \int_0^t \left\langle \omega^+_s, \left(\sigma^k \cdot \nabla \sigma^k \right) \cdot \nabla \phi \right\rangle ds + \frac{1}{2} \int_0^t \left\langle \omega^+_s, \left(\sigma^{k}\right)^T \left(\nabla^2 \phi \right)\sigma^k  \right\rangle ds,
\end{aligned}
\end{equation}
and,
\begin{equation}
\begin{aligned}
&\left\langle \omega^-_t ,\phi \right\rangle + \int_0^t \left\langle \omega^-_s, K \star \left( \omega^+_s - \omega^-_s \right) \cdot\nabla\phi \right\rangle ds =\nu \int_0^t \left\langle \omega^-_s, \Delta \phi \right\rangle ds + \\
& \int_0^t \sum_{k\in K}
\left\langle \omega^-_s, \sigma_{k} \cdot\nabla\phi \right\rangle dW^k_s + \frac{1}{2} \int_0^t \left\langle \omega^-_s, \left(\sigma^k \cdot \nabla \sigma^k \right) \cdot \nabla \phi \right\rangle ds + \frac{1}{2} \int_0^t \left\langle \omega^-_s, \left(\sigma^{k}\right)^T \left(\nabla^2 \phi \right)\sigma^k  \right\rangle ds.
\end{aligned}
\end{equation} 
\end{definition}

We will make the following assumptions:
\begin{ass}
\label{ass:assprinci}
Let $ \beta \in (0,1)$, $p > 2$ and $\alpha \in (2/p, 1)$. Assume that:
\begin{enumerate}
    \item $V^N \in C^{\infty}(\T^2)$ and it is an approximation of the identity.
    \item $V^N(x)=N^{2\beta}V(N^{\beta}x)$ for a fixed $V$ (an explicit example is given in the Appendix).
    \item For all $q > 0$:
    \[\sup\limits_{N \in \mathbb{N}} \mathbb{E}[ \| V^N \star \mu_0^{N, \pm} \|^q_{\HH^{\alpha}_p(\T^2)}] < +\infty.\]
    \item $\omega_0 \in \mathbb{L}^{\infty}(\T^2)$ is such that $\{\mu_0^{N, +}\}_{N \in \N}$ weakly converges to $\omega_0^{+}(x)dx$ in probability, and $\{\mu_0^{N, -}\}_{N \in \N}$ weakly converges to $\omega_0^{-}(x)dx$ in probability.
    \item $0<\beta < \frac{1}{6 +2\alpha - \frac{4}{p}} < \frac{1}{4}$.
    \item $\sum\limits_{k \in \N} \sigma^k(x)^T \sigma^k(x)$ is such that $\nabla^2 \left( \sum\limits_{k \in \N} \sigma^k(x)^T \sigma^k(x) \right) = 0$ for all $x \in \T^2$.
    \end{enumerate}
\end{ass}

The last point in Assumption \ref{ass:assprinci} is not unusual in the context of stochastic fluid dynamics. As an example, a stronger version of it has been assumed in \cite{brzeniak2016}.

Applying the It\^{o} formula to $\phi(X^{i, N, \pm}_t)$ and then summing over $i$ one has:
\begin{equation}
\label{eq:empirical_process_test_function}
\begin{aligned}
\left\langle \mu_t^{N, +}, \phi \right\rangle = & \left\langle \mu_0^{N, +}, \phi \right\rangle + \int_0^t \left\langle \mu_s^{N, +}, F\left(K \star  V^N \star \mu_s^{N} \right) \cdot \nabla \phi \right\rangle ds + \\
&+ \frac{\sqrt{2\nu}}{N}\sum\limits_{i=1}^N \int_0^t \nabla \phi (X^{i,N, +}_s) \cdot d B^{i, +}_s + \int_0^t \left\langle \mu_s^{N, +} , \nu \Delta \phi \right\rangle ds +  \int_0^t \left\langle \mu_s^{N, +} , \sum\limits_{k \in \N} \left( \sigma^k \right)^T \left( \nabla^2 \phi \right) \sigma^k \right\rangle ds\\
&+ \sum\limits_{k \in \mathbb{N}} \int_0^t \left\langle \mu_s^{N, +}, \sigma_k \cdot \nabla \phi \right\rangle dW^k_s + \sum\limits_{k \in \N} \int_0^t \left\langle \mu_s^{N, +}, \left(\sigma_k \cdot \nabla \sigma_k\right) \cdot \nabla \phi \right\rangle ds.
\end{aligned}
\end{equation}

The same equation is satisfied by $\mu^{N, -}_t$ with the same exact computations.

Define $g_t^{N, \pm}(x) := V^N \star \mu_t^{N, \pm}(x)$, and as for $\mu^{N}$, we will write $g^N_t:=g^{N, +}_t - g^{N, -}_t$. Choosing as a test function $\phi_x(y):=V^N(x-y)$ in \eqref{eq:empirical_process_test_function} we obtain the following equation for $g^{N, +}_t$, for a fixed $x \in \T^2$:
\begin{equation}
\label{eq:g}
\begin{aligned}
g_t^{N, +}(x) = & g_0^{N, +}(x) + \int_0^t \left\langle \mu_s^{N, +}, F\left(K \star V^N \star \mu^N_s\right) \cdot \nabla V^N(x-\cdot) \right\rangle ds + \\
&+ \frac{\sqrt{2\nu}}{N}\sum\limits_{i=1}^N \int_0^t \nabla V^N(x-X^{i,N, +}_t) d B^{i, +}_s + \int_0^t \nu \Delta g_s^{N, +}(x) ds + \int_0^t \sum\limits_{k \in \mathbb{N}} \left( \sigma^k \right)^T \nabla^2 g_s^{N, +} \sigma^k (x) ds \\
&+ \sum\limits_{k \in \mathbb{N}} \int_0^t \left\langle \mu_s^{N, +}, \nabla V^N(x-\cdot) \sigma_k \right\rangle dW^k_s + \\
&+ \sum\limits_{k \in \mathbb{N}} \int_0^t \left\langle \mu_s^{N, +}, \left( \nabla\sigma_k \cdot \sigma_k \right) \cdot \nabla V^N(x-\cdot) \right\rangle ds,
\end{aligned}
\end{equation}
and with the same exact computations one shows that $g^{N, -}_t$ satisfies the same equation.

\begin{remark}
Let us justify the term:
\[
\int_0^t \sum\limits_{k \in \mathbb{N}} \left( \sigma^k \right)^T \nabla^2 g_s^{N, +} \sigma^k (x) ds.
\]
One can write:
\[
\begin{aligned}
& \int_0^t \left\langle \mu_s^{N, +} , \sum\limits_{k \in \N} \left( \sigma^k \right)^T \left( \nabla^2 V^N(x-\cdot) \right) \sigma^k \right\rangle ds = \int_0^t \left\langle \mu_s^{N, +} , \nabla^2 \sum\limits_{k \in \N} \left( \sigma^k \right)^T \sigma^k V^N(x-\cdot) \right\rangle ds \\
&= \int_0^t \left\langle \mu_s^{N, +} , \nabla^2 \left( V^N(x-\cdot) \sum\limits_{k \in \N} \left( \sigma^k \right)^T \sigma^k \right) \right\rangle ds
\end{aligned}
\]
and since:
\[
\sum\limits_{k \in \N} \left( \sigma^k \right)^T \sigma^k
\]
is constant by Assumption \ref{ass:assprinci} one has:
\[
\begin{aligned}
&\int_0^t \left\langle \mu_s^{N, +} , \sum\limits_{k \in \N} \left( \sigma^k \right)^T \left( \nabla^2 V^N(x-\cdot) \right) \sigma^k \right\rangle ds = \int_0^t \sum\limits_{k \in \N} \left( \sigma^k \right)^T \sigma^k\left\langle \mu_s^{N, +} , \nabla^2 \left( V^N(x-\cdot) \right) \right\rangle \\
&=\int_0^t \sum\limits_{k \in \N} \left( \sigma^k \right)^T \sigma^k \nabla^2 g^{N, +}(x) = \int_0^t \sum\limits_{k \in \mathbb{N}} \left( \sigma^k \right)^T \nabla^2 g_s^{N, +} \sigma^k (x) ds.
\end{aligned}
\]
\end{remark}

Let \(\phi \in C^{\infty}(\mathbb{T}^2)\). By taking the scalar product between each term in \eqref{eq:g} and \(\phi(x)\), and integrating over \(\T^2\), we obtain:
\begin{equation}
\label{eq:g_with_test_function_first_version}
\begin{aligned}
\left\langle g_t^{N, +}, \phi \right\rangle = & \left\langle g_0^{N, +}, \phi \right\rangle + \int_0^t \left\langle \mu_s^{N, +}, F\left(K \star V^N \star \mu_s^{N}\right) \cdot \nabla \left( V^N \star \phi \right) \right\rangle ds + \\
&+ \frac{1}{N}\sum_{i=1}^N \int_0^t \nabla \left( V^N \star \phi \right)(X^{i,N, +}_t) d B^{i, +}_s + \int_0^t \left\langle \nu \Delta g_s^{N, +}, \phi \right\rangle ds + \int_0^t \left\langle \sum_{k} \left( \sigma^k \right)^T \nabla^2 g_s^{N, +} \sigma^k, \phi \right\rangle ds + \\
&+ \sum_{k \in \mathbb{N}} \int_0^t \left\langle \left( \sigma^k \mu_s^{N, +} \right) \star V^N, \nabla \phi \right\rangle dW^k_s + \sum_{k \in \mathbb{N}} \int_0^t \left\langle \left( \left( \sigma^k \cdot \nabla \sigma^k \right) \mu_s^{N, +} \right)\star V^N, \nabla \phi \right\rangle ds. 
\end{aligned}
\end{equation}
Again, same exact computations and one can show that the same holds for $g^{N, -}_t$.

We are now ready to state our main result:
\begin{theorem}
\label{th:main_theorem}
    Assume Assumption \ref{ass:assprinci}. Then, \((g^{N, +}, g_t^{N, -})\) converges to the solution of the system \eqref{eq:main_equation} in the sense of definition \ref{def:solution_of_system}, in the strong topology of \(C([0, T]; \HH^{\eta}_p(\T^2))^2\), for any \(\eta \in \left(\frac{2}{p}, \alpha\right)\), in probability. In other words, for any $\varepsilon > 0$:
    \[
    \prob \left[ \norm{(g^{N, +}, g_t^{N, -}) - (\omega^+, \omega^-)}_{C([0, T]; \HH^{\eta}_p(\T^2))^2} > \varepsilon \right] \xrightarrow{N \to +\infty} 0.
    \]
\end{theorem}

%%%%%%%%%%%%%%%%%%%%%%%%%%%%%%%%%%%%%%%%%%%%%%%%%%%%%

%%%%%%%%%%%%%%%%%%%%%%%%%%%%%%%%%%%%%%%%%%%%%%%%%%%%%

%%%%%%%%%%%%%%%%%%%%%%%%%%%%%%%%%%%%%%%%%%%%%%%%%%%%%
\section{Criterion of compactness}
\label{sec:compactness}

This section is dedicated to proving the following theorem:

\begin{theorem}
    Let $\eta \in (2/p, \alpha)$. Let \(L^{N, \pm}\) represent the laws of \(g^{N, \pm} : \Xi \to \mathfrak{F} = C([0,T], \HH^{\eta}_p(\T^2))\). Then, the family \((L^{N, \pm})_{N}\) is relatively compact.
\end{theorem}

\begin{proof}
Let \(X = \HH^{\alpha}_p(\T^2)\), \(B = \HH^{\eta}_p(\T^2)\), and \(Y = \HH^{-2}_2(\T^2)\), with \(\alpha > \frac{2}{p}\) and \(\frac{2}{p} < \eta < \alpha\). We note that \(X \subset B \subset Y\) exhibits a compact dense embedding, which follows from the Kondrachov embedding theorem. For further details, refer to \cite{belyaev2022multipliers}. Additionally, for any \(f \in X\):
\[
\|f\|_B \leq C_R \|f\|_X^{1-\theta} \|f\|_Y^{\theta}
\]
where \(\theta = \frac{\alpha - \eta}{2+\alpha}\). Maintaining the notation from \cite{simon}, we define \(s_0 := 0\), \(s_1 := \gamma \in (0,\frac{1}{2})\), and select \(q, q' \geq 2\) such that:
\[
s_1 q' = \gamma q' > 1 \quad \text{and} \quad s_{\theta} := \theta s_1 = \theta \gamma > \frac{1-\theta}{q} + \frac{\theta}{q'}.
\]
Consequently, Corollary 9 of \cite{simon} tells us that the space 
\[
\LL^q([0,T], X) \cap \W^{\gamma, q'}([0,T], Y)
\]
is relatively compact in \(C([0,T],B)\). Therefore, for any \(\gamma \in (0,\frac{1}{2})\), for \(\alpha\) as specified in Assumption \ref{ass:assprinci}, and for \(q, q' \geq 2\), we consider the space:

\[
\mathfrak{Y}_0 = \LL^q([0,T], \HH^{\alpha}_p(\T^2)) \cap \W^{\gamma, q'}([0,T], \HH^{-2}_{2}(\T^2))
\]
and the space:
\begin{equation}
\label{definition_space_y}
\mathfrak{Y} := C([0,T], \HH^{\eta}_p(\T^2)).
\end{equation}
We conclude that \(\mathfrak{Y}_0\) is compactly embedded in \(\mathfrak{Y}\) for any \(\frac{2}{p} < \eta < \alpha\).

Applying Chebyshev's inequality, we obtain:
\[
\prob \left[\| g^{N, \pm}_{\cdot }\|^2_{\mathfrak{Y}_0} > R\right] \leq \frac{\E\left[\|g_{\cdot}^{N, \pm}\|^2_{\mathfrak{Y}_0}\right]}{R}
\]
for any \(R > 0\).

From Proposition \ref{prop1}:
\[
\E \bigg[ \bigg\| (I-A)^{\frac{\alpha}{2}} g_t^{N, \pm} \bigg\|_{\LL^p(\T^2)}^q \bigg] \leq C, \quad \text{ for all } t \in [0,T],
\]
from which we can easily conclude that 
\[
\E \bigg[ \| g^{N, \pm} \|^2_{\LL^q([0,T], \HH^{\alpha}_p(\T^2))} \bigg] \leq C.
\]
From Proposition \ref{prop2}
\[
\E \bigg[ \| g^{N, \pm} \|^2_{\W^{\gamma, q'}([0,T], \HH^{-2}_{2}(\T^2))} \bigg] \leq C.
\]
Thus, we obtain:
\[
\prob \left[\| g_{\cdot}^{N, \pm} \|^2_{\mathfrak{Y}} > R\right] \leq \prob \left[\| g_{\cdot}^{N, \pm} \|^2_{\mathfrak{Y}_0} > R\right] \leq \frac{C}{R},
\]
from which we deduce that the sequence of laws $\{L_N^{\pm}\}_{N \in \mathbb{N}}$ is tight in \(\mathfrak{Y}\) and the Theorem is proven.
\end{proof}

\begin{remark}
From the proof of this Theorem, we also get that $g^{N, \pm}$ is uniformly bounded in:
\[
\LL^{2}\left( \Xi; C\left([0,T]; \HH^{\eta}_p(\T^2) \right) \right)
\]
\end{remark}

%%%%%%%%%%%%%%%%%%%%%%%%%%%%%%%%%%%%%%%%%%%%%%%%%%%%%

%%%%%%%%%%%%%%%%%%%%%%%%%%%%%%%%%%%%%%%%%%%%%%%%%%%%%

%%%%%%%%%%%%%%%%%%%%%%%%%%%%%%%%%%%%%%%%%%%%%%%%%%%%%

%%%%%%%%%%%%%%%%%%%%%%%%%%%%%%%%%%%%%%%%%%%%%%%%%%%%%
\section{Passing to the limit}
\label{sec:passing-to-the-limit}

In this section, we prove the following result:

\begin{theorem}
    Assume that \(g^{N, +}\) converges almost surely to a random variable \(\omega^+\) in \(C([0,T]; \HH^{\eta}_p(\T^2))\), and \(g^{N, -}\) converges almost surely to another random variable \(\omega^{-}\) in \(C([0,T]; \HH^{\eta}_p(\T^2))\). Then, \((\omega^{+}, \omega^{-})\) is a solution in the sense of Definition \ref{def:solution_of_system}.
\end{theorem}
We will prove the result for $\omega^+$ only, as the case for $\omega^-$ follows performing the same exact computations. For simplicity, we will omit the sign $+$ for the rest of the proof.

\begin{remark}
Consider the equation \eqref{eq:g_with_test_function_first_version}, and assume that \(g^{N}_t \xrightarrow[N \to \infty]{} \omega\) in \(C([0,T]; \HH^{\eta}_p(\T^2))\) almost surely. Let \(\phi \in C^{\infty}(\T^2)\). It is straightforward to notice that:
\[
\left\langle g_t^N, \phi\right\rangle \underset{N \rightarrow \infty}{\longrightarrow}\langle\omega_t, \phi\rangle, \quad\left\langle g_0^N, \phi\right\rangle \underset{N \rightarrow \infty}{\longrightarrow}\left\langle\omega_0, \phi\right\rangle.
\]
and
\[
\begin{aligned}
\mathbb{E}\left[\left|\frac{1}{N} \sum_{i=1}^N \int_0^t \nabla\left(V^N \star \phi\right)\left(X_s^{i, N}\right) d B_s^i\right|^2\right] & = \frac{1}{N^2} \sum_{i=1}^N \int_0^t \mathbb{E}\left[\left|\nabla\left(V^N \star \phi\right)\left(X_s^{i, N}\right) \right|^2\right] d s \\
& \leq \frac{t}{N} \cdot \|\nabla \phi\|_{\mathbb{L}^{\infty}}^2 \underset{N \rightarrow \infty}{\longrightarrow} 0.
\end{aligned}
\]
This follows from the fact that:
\[
\nabla\left(V^N \star \phi\right)\left(X_s^{i, N}\right) = V^N \star \nabla \phi\left(X_s^{i, N}\right),
\]
and for all \(x \in \T^2\):
\[
\left| V^N \star \nabla \phi (x) \right|^2 \leq \|\nabla \phi\|_{\mathbb{L}^{\infty}}^2 \left| \int_{\T^2} V^N(y) dy \right|^2 = \|\nabla \phi\|_{\mathbb{L}^{\infty}}^2.
\]
\end{remark}

Before proceeding with our proof, let us recall the following result (see equation (1.21) of \cite{nualart}):

\begin{lemma}
\label{lemma:ito_isometry_prob}
For any $\epsilon >0 $ and any $K>0$:
    \begin{equation}
        \prob \left[\left|\int_0^t f(s) d W_s\right|>\epsilon\right] \leq \prob \left[ \int_0^t f(s)^2 d s>K\right]+\frac{K}{\epsilon^2}
    \end{equation}
\end{lemma}

\begin{lemma}
Let $\phi$ be fixed as always. If $g^{N, +} \xrightarrow[N \to \infty]{} \omega^{+}$ and $g^{N, -} \xrightarrow[N \to \infty]{} \omega^{-}$ in $C([0,T]; \HH^{\eta}_p(\T^2))$ in probability, then for every $t \in [0,T]$:
\[
\sum\limits_{k \in \mathbb{N}} \int_0^t \left\langle \left( \sigma^k \mu_s^N \right) \star V^N , \nabla \phi \right\rangle dW^k_s \xrightarrow[N \to \infty]{\prob} \sum\limits_{k \in \mathbb{N}} \int_0^t \left\langle \omega_s , \sigma^k \cdot \nabla \phi \right\rangle dW^k_s.
\]
\end{lemma}

\begin{proof} First of all let us notice that:
\[
\sum\limits_{k \in \mathbb{N}} \int_0^t \left\langle \left( \sigma^k \mu_s^N \right) \star V^N , \nabla \phi \right\rangle dW^k_s = \sum\limits_{k \in \mathbb{N}} \int_0^t \left\langle \mu_s^N , \left( V^N \star \nabla \phi \right) \cdot \sigma^k \right\rangle dW^k_s.
\]
This is because:
\[
\begin{aligned}
    \left\langle \left( \sigma^k \mu_s^N \right) \star V^N , \nabla \phi \right\rangle & = \int_{\T^2} \frac{1}{N}\sum\limits_{i=1}^N \sigma^k(X^{i,N}_s)V^N(x-X^{i,N}_s) \nabla \phi(x) dx = \\
    & = \frac{1}{N}\sum\limits_{i=1}^N \sigma^k(X^{i,N}_s) \int_{\T^2} V^N(X^{i,N}_s-x) \nabla \phi(x) dx = \\
    & = \frac{1}{N}\sum\limits_{i=1}^N \sigma^k(X^{i,N}_s) \left( V^N \star \nabla \phi \right) (X^{i,N}_s) = \left\langle \mu_s^N, \left( V^N \star \nabla \phi \right) \cdot \sigma^k \right\rangle.
\end{aligned}
\]
Then, it is sufficient to show that:
\begin{equation}
\label{eq:limit_noise_term}
\prob \left[ \left|  \sum\limits_{k \in \mathbb{N}} \int_0^t \left\langle \mu_s^N , \left( V^N \star \nabla \phi \right) \cdot \sigma^k \right\rangle dW^k_s -\sum\limits_{k \in \mathbb{N}} \int_0^t \left\langle g^N_s , \left( V^N \star \nabla \phi \right) \cdot \sigma^k \right\rangle dW^k_s \right| > \epsilon \right] \xrightarrow[N \to \infty]{} 0
\end{equation}
In fact, assuming \eqref{eq:limit_noise_term} and applying Lemma \ref{lemma:ito_isometry_prob} with $K=\epsilon^3$, we deduce that 
\begin{equation}
\label{eq:prob_ito_integral}
\begin{aligned}
& \prob \left[ \left| \int_0^t \left\langle g^N_s , \sigma^k \cdot \left( V^N \star \nabla \phi \right) \right\rangle dW^k_s -  \int_0^t \left\langle \omega_s , \sigma^k \cdot \nabla \phi \right\rangle dW^k_s \right|>\epsilon \right] = \\
& \underset{\text{Lemma}}{\leq} \prob \left[ \int_0^t \left| \left\langle g^N_s , \sigma^k \cdot \left( V^N \star \nabla \phi \right) \right\rangle - \left\langle \omega_s , \sigma^k \cdot \nabla \phi \right\rangle \right|^2 ds > \epsilon^3 \right]  + \epsilon \leq \\
& \leq \prob \left[ \int_0^t \left| \left\langle g^N_s , \sigma^k \cdot \left( V^N \star \nabla \phi - \nabla \phi \right) \right\rangle \right|^2 ds > \frac{\epsilon^3}{2}\right] + \prob \left[ \int_0^t \left| \left\langle g^N_s - \omega_s , \sigma^k \cdot \nabla \phi \right\rangle \right|^2 ds > \frac{\epsilon^3}{2} \right] + \epsilon.
\end{aligned}
\end{equation}

Let us prove now that the last two terms in the equation \eqref{eq:prob_ito_integral} converge to $0$. First of all:
\[
\begin{aligned}
    \prob \left[ \int_0^t \left| \left\langle g^N_s , \sigma^k \cdot \left( V^N * \nabla \phi - \nabla \phi \right) \right\rangle \right|^2 ds > \frac{\epsilon}{2}\right] & \leq \prob \left[ \left\| \sigma^k \cdot \left( V^N * \nabla \phi - \nabla \phi \right) \right\|^2_{\HH^{-\eta}_p(\T^2)} \int_0^t \left\|g^N_s\right\|^2_{\HH^{\eta}_p(\T^2)}ds > \frac{\epsilon}{2}\right] \\
\end{aligned}
\]
and using $C([0,T], \HH^{\eta}_p(\T^2)) \subset \LL^2([0,T], \HH^{\eta}_p(\T^2))$ it is sufficient to bound:
\[
\begin{aligned}
    & \prob \left[ \left\| \sigma^k \cdot \left( V^N * \nabla \phi - \nabla \phi \right) \right\|^2_{\HH^{-\eta}_p(\T^2)} \left\|g^N_s\right\|^2_{C\left([0,T], \HH^{\eta}_p(\T^2)\right)}ds > \frac{\epsilon}{2 C }\right] \leq \\
    & \leq \prob \left[ \left\| \sigma^k \cdot \left( V^N * \nabla \phi - \nabla \phi \right) \right\|^2_{\HH^{-\eta}_p(\T^2)} \left\|g^N_s-\omega \right\|^2_{C\left([0,T], \HH^{\eta}_p(\T^2)\right)}ds > \frac{\epsilon}{2 C }\right] \\
    & + \prob \left[ \left\| \sigma^k \cdot \left( V^N * \nabla \phi - \nabla \phi \right) \right\|^2_{\HH^{-\eta}_p(\T^2)} \left\|\omega\right\|^2_{C\left([0,T], \HH^{\eta}_p(\T^2)\right)}ds > \frac{\epsilon}{2 C }\right] \xrightarrow[N \to \infty]{} 0.
\end{aligned}
\]
The last term goes to zero because of the assumption that $\left\|g^N -\omega\right\|^2_{C\left([0,T], \HH^{\eta}_p(\T^2)\right)} \to 0$ in probability and because $\left\| \sigma^k \cdot \left( V^N \star \nabla \phi - \nabla \phi \right) \right\|^2_{\HH^{-\eta}_p(\T^2)} \xrightarrow[N \to \infty]{} 0$ as a consequence of $V^N$ being an approximation of the identity.

For the second term in equation \eqref{eq:prob_ito_integral}, notice that:
\[
\begin{aligned}
    & \prob \left[ \int_0^t \left| \left\langle g^N_s - \omega_s , \sigma^k \cdot \nabla \phi \right\rangle \right|^2 ds > \frac{\epsilon}{2}\right] \leq \prob \left[ \left\| \sigma^k \cdot \nabla \phi \right\|^2_{\HH^{-\eta}_p(\T^2)} \int_0^t \left\|g^N_s - \omega_s\right\|^2_{\HH^{\eta}_p(\T^2)}ds > \frac{\epsilon}{2}\right] \\
\end{aligned}
\]
As above, because $C([0,T], \HH^{\eta}_p(\T^2)) \subset \LL^2([0,T], \HH^{\eta}_p(\T^2))$ one has:
\[
\begin{aligned}
    & \prob \left[ \left\| \sigma^k \cdot \nabla \phi \right\|^2_{\HH^{-\eta}_p(\T^2)} \left\|g^N_s-\omega_s\right\|^2_{C\left([0,T], \HH^{\eta}_p(\T^2)\right)}ds > \frac{\epsilon}{2 C }\right] \\
    & = \prob \left[ \left\|g^N_s-\omega_s\right\|^2_{C\left([0,T], \HH^{\eta}_p(\T^2)\right)}ds > \frac{\epsilon}{2 C \left\| \sigma^k \cdot \nabla \phi \right\|^2_{\HH^{-\eta}_p(\T^2)}}\right] \xrightarrow[N \to \infty]{} 0.
\end{aligned}
\]

Now, to prove \eqref{eq:limit_noise_term}, we apply Lemma \ref{lemma:ito_isometry_prob} to get:

\begin{equation}
\begin{aligned}
    & \prob \left[ \left|  \sum\limits_{k \in \mathbb{N}} \int_0^t \left\langle \mu_s^N , \left( V^N \star \nabla \phi \right) \cdot \sigma^k \right\rangle dW^k_s -\sum\limits_{k \in \mathbb{N}} \int_0^t \left\langle g^N_s , \left( V^N \star \nabla \phi \right) \cdot \sigma^k \right\rangle dW^k_s \right| > \epsilon \right] \leq \\
    & \leq \prob \left[ \sum\limits_{k \in \mathbb{N}} \int_0^t \left| \left\langle \mu_s^N , \left( V^N \star \nabla \phi \right) \cdot \sigma^k \right\rangle - \left\langle g^N_s , \left( V^N \star \nabla \phi \right) \cdot \sigma^k \right\rangle \right|^2 ds > \epsilon^3 \right] + \epsilon,
    \end{aligned}
\end{equation}
and therefore it is sufficient to show that the right hand side of the following inequality converges to $0$:
\begin{equation}
\label{eq:limit_stochastic_term}
\begin{aligned}
& \left| \left\langle \mu_s^N , \left( V^N \star \nabla \phi \right) \cdot \sigma^k \right\rangle - \left\langle \mu_s^N \star V^N, \left( V^N \star \nabla \phi \right) \cdot \sigma^k \right\rangle \right| \\
&\leq \sup\limits_{x \in \T^2} \left| \left( V^N \star \nabla \phi \right) \cdot \sigma^k(x) - \left( \left( V^N \star \nabla \phi \right) \cdot \sigma^k \right) \star V^N (x) \right|.
\end{aligned}
\end{equation}
Now:
\[
\begin{aligned}
& \left| \sigma^k(x) \cdot \nabla\left(V^N \star \phi\right)(x)-\left( \left(\sigma^k \cdot \nabla\left(V^N \star \phi\right)\right) \star V^N\right)(x) \right| \\
& \stackrel{\left(\int V^N=1\right)}{\leq} \int_{\T^2} V(y)\left\|\nabla\left(V^N \star \phi\right)(x)\right\|\left\|\sigma^k(x)-\sigma^k\left(x-\frac{y}{N^\beta}\right)\right\| d y \\
& +\int_{\T^2} V(y)\left\|\nabla\left(V^N \star \phi\right)(x)-\nabla\left(V^N \star \phi\right)\left(x-\frac{y}{N^\beta}\right)\right\|\left\| \sigma^k(x)\right\| d y \\
& \leq \int_{\T^2} V(y)\left\|\nabla\left(V^N \star \phi\right)(x)\right\|\left\|\sigma^k(x)-\sigma^k\left(x-\frac{y}{N^\beta}\right)\right\| d y +\frac{C}{N^\beta} \int_{\T^2} V(y)\|y\| d y  \leq \\
& \leq \frac{1}{N^{\tilde{\eta} \beta}} \sup _{x, y \in \T^2} \frac{\left\|\sigma^k(x)-\sigma^k(y)\right\|}{\|x-y\|^{\tilde{\eta}}} \int_{\T^2} V(y)\|y\|^{\tilde{\eta}} d y + \frac{1}{N^\beta} \int_{\T^2} V(y)\|y\| d y \xrightarrow[N \to \infty]{} 0,
\end{aligned}
\]
for any $0<\tilde{\eta}<1$.
Now, the term:
\[
\sup _{x, y \in \T^2} \frac{\left\|\sigma^k(x)-\sigma^k(y)\right\|}{\|x-y\|^{\tilde{\eta}}} < +\infty
\]
is finite because $\sigma^k \in C^{\infty}(\T^2)$ and because we assumed a uniform bound on the Lipschitz constants of $\sigma^k$.
\end{proof}

\begin{lemma}
Let $\phi$ be fixed as always. If $g^{N, +} \xrightarrow[N \to \infty]{} \omega^{+}$ and $g^{N, -} \xrightarrow[N \to \infty]{} \omega^{-}$ in $C([0,T]; \HH^{\eta}_p(\T^2))$ in probability, then it holds that:
\begin{equation}
\label{eq_limit}
\lim _{N \rightarrow \infty} \int_0^t\left\langle \mu_s^{N, +}, F\left( K \star g^N_s \right) \cdot \nabla\left(V^N \star \phi\right)\right\rangle d s=\int_0^t \int_{\T^2} \omega^{+}_s(x) F \left( K \star \omega_s(x) \right) \cdot \nabla \phi(x) d x d s .
\end{equation}
\end{lemma}
\begin{proof}
We have that 
\begin{equation}
\label{eq:making_limit_easier}
\begin{aligned}
& \left|\left\langle \mu_s^{N, +}, F \left( K \star g_s^{N} \right) \cdot \nabla\left(V^N \star \phi\right)\right\rangle-\left\langle g_s^{N, +}, F\left(K \star g_s^{N}\right) \cdot \nabla\left(V^N \star \phi\right)\right\rangle\right| \\
& \quad \leq \sup_{x \in \mathbb{T}^2} \left| F\left( K \star g_s^{N}\right) \cdot \nabla\left(V^N \star \phi\right)(x)-\left( \left(F\left( K \star g_s^{N} \right) \cdot \nabla\left(V^N \star \phi \right) \right) \star V^N \right) (x) \right|
\end{aligned}
\end{equation}
because $g_t^{N, +}= V^N \star \mu_t^{N, +}$.

We start by controlling the last term:
\begin{equation}
\begin{aligned}
& \left| F\left( K \star g_s^{N}\right)(x) \cdot \nabla\left(V^N \star \phi\right)(x)-\left( \left(F\left( K \star g_s^{N} \right) \cdot \nabla\left(V^N \star \phi\right)\right) \star V^N\right)(x) \right| \\
& \stackrel{\left(\int V=1\right)}{\leq} \int_{\T^2} V(y)\left\|\nabla\left(V^N \star \phi\right)(x)\right\|\left\|F\left( K \star g_s^{N}\right)(x)-F\left( K \star g_s^{N}\right)\left(x-\frac{y}{N^\beta}\right)\right\| d y \\
& +\int_{\T^2} V(y)\left\|\nabla\left(V^N \star  \phi\right)(x)-\nabla\left(V^N \star \phi\right)\left(x-\frac{y}{N^\beta}\right)\right\|\left\| F \left( K \star g_s^{N}  \right)(x)\right\| d y \\
& \stackrel{\left(F \in \operatorname{Lip} \cap L^{\infty}\right)}{\leq} C \int_{\T^2} V(y)\left\|\nabla\left(V^N \star  \phi\right)(x)\right\|\left\|\left( K \star g_s^{N}\right)(x)-\left( K \star g_s^{N} \right) \left(x-\frac{y}{N^\beta}\right)\right\| d y \\
& +\frac{C}{N^\beta} \int_{\T^2} V(y)\|y\| d y \\
& \leq \frac{C}{N^{\tilde{\eta} \beta}} \sup _{x, y \in \T^2} \frac{\left\|K \star g_s^{N, +}(x)-K \star g_s^{N, +}(y)\right\|}{\|x-y\|^{\tilde{\eta}}} + \frac{C}{N^{\tilde{\eta} \beta}} \sup _{x, y \in \T^2} \frac{\left\|K \star g_s^{N, -}(x)-K \star g_s^{N, -}(y)\right\|}{\|x-y\|^{\tilde{\eta}}} + \int_{\T^2} V(y)\|y\|^{\tilde{\eta}} d y \\
& +\frac{C}{N^\beta} \int_{\T^2} V(y)\|y\| d y, \\
&
\end{aligned}
\end{equation}
For the first inequality we exploited that $\int V = 1$ and we added and subtracted the term:
\[
\nabla\left(V^N \star \phi\right)(x) \cdot F\left(K \star g_s^N\right)(x)
\]

Therefore we have obtained
$$
\left.\mid F\left(K \star g_s^N\right)(x) \cdot \nabla\left(V^N \star \phi\right)(x)-\left(F\left(K \star g_s^N\right) \cdot \nabla\left(V^N \star \phi\right)\right) \star V^N\right)(x) \mid \leq \frac{C}{N^{\tilde{\eta} \beta}},
$$
where the constant $C$ depends on $\left\|K \star g^{N, \pm}\right\|_{C^{\tilde{\eta}}\left(\mathbb{R}^2\right)}$. Thus,
\begin{equation}
\begin{aligned}
\lim _{N \rightarrow \infty} \int_0^t\left\langle \mu_s^N, F\left(K \star g_s^N\right)\right. & \left.\cdot \nabla\left(V^N \star \phi\right)\right\rangle d s \\
& =\lim _{N \rightarrow \infty} \int_0^t\left\langle g_s^N, F\left(K \star g_s^N\right) \cdot \nabla\left(V^N \star \phi\right)\right\rangle d s \\
& =\lim _{N \rightarrow \infty} \int_0^t \int_{\mathbb{R}^2} g_s^N(x) F\left(K \star g_s^N\right)(x) \cdot \nabla\left(V^N \star \phi\right)(x) d x d s \\
& =\int_0^t \int_{\mathbb{R}^2} \omega(s, x) F(K \star \omega_s)(x) \cdot \nabla \phi(x) d x d s,
\end{aligned}
\end{equation}

where in the last equality we used that $g_s^N \rightarrow \omega$ with respect to the strong topology of $\mathbb{L}^2\left([0, T]; C\left(\mathbb{T}^2\right)\right)$ in probability. In fact, because of section 2.8.1 of \cite{triebel}, $\HH^{\eta}_p\left(\T^2\right)) \subset C(\T^2)$, and because if $g^N$ converges strongly to $\omega$ in $C([0,T], \HH^{\eta}_p\left(\T^2\right))$, it converges strongly also in $\mathbb{L}^2\left([0, T]; C\left(\mathbb{T}^2\right)\right)$. We have proven equation \eqref{eq_limit}.
\end{proof}

%%%%%%%%%%%%%%%%%%%%%%%%%%%%%%%%%%%%%%%%%%%%%%%%%%%%%

%%%%%%%%%%%%%%%%%%%%%%%%%%%%%%%%%%%%%%%%%%%%%%%%%%%%%
\section{Lifting the truncation}
\label{sec:lifting}
\subsection{Existence and uniqueness of the 2d-SNS}

Let $f_R$ be equal to 1 on $[0, R]$, equal to 0 on $[R+1, \infty)$, and decreasing on $[R, R+1]$, for arbitrary $R>0$. We then have the following Proposition:

\begin{prop}
Let $h$ be a positive integer. Let $\omega_0 \in \HH^h_2(\T^2)$, then the following equation:
\begin{equation}
\label{eq:SNS}
d \omega_t^R+f_R\left(\left\|\omega_t^R\right\|_{h-1,2}\right) u_t^R \cdot \nabla \omega_t^R d t+\sum_{k = 1}^{\infty} \sigma_k \cdot \nabla \omega_t^R \circ d W_t^k= \nu \Delta \omega^{R}_t dt
\end{equation}
admits a unique global $\left(\mathcal{F}_t\right)_t$-adapted solution $\omega^R=\left\{\omega_t^R, t \in[0, \infty)\right\}$ with values in the space $C\left([0, \infty), \HH^{h-1}_2(\T^2)\right)$. In particular, if $h \geq 5$, the solution is classical.
\end{prop}

The proof of this Proposition follows from Theorem 1 and 2 in \cite{Stochastic_Evolution_Systems}.

\begin{prop}
The solution to the equation \eqref{eq:SNS} is global.
\end{prop}
\begin{proof}
Define $\tau_R(\omega):=\inf _{t \geq 0}\left\{\left\|\omega_t^R\right\|_{h-1,2} \geq R\right\}$. Observe that on $\left[0, \tau_R\right], f_R\left(\left\|\omega_t^R\right\|_{k-1,2}\right)=1$, and therefore, on $\left[0, \tau_R\right]$ the solution of the truncated equation \eqref{eq:SNS} is, in fact a solution of \eqref{eq:main_equation}. It therefore makes sense to define the process $\omega=\left\{\omega_t, t \in[0, \infty)\right\}$ where $\omega_t=\omega_t^R$ for $t \in\left[0, \tau_R\right]$. This definition is consistent as, following the uniqueness property of the solution of the truncated equation (see \cite{Stochastic_Evolution_Systems}), $\omega_t^R=\omega_t^{R^{\prime}}$ for $t \in\left[0, \tau_{\min \left(R, R^{\prime}\right)}\right]$.

The process $\omega$ defined this way is a solution of the \eqref{eq:main_equation} on the interval $\left[0, \sup _R \tau_R\right)$. To obtain a global solution we need to prove that $\sup _{R>0} \tau_R =\infty$ almost surely. Let $\mathscr{A}:=\left\{\omega \in \Omega \mid \sup _{R>0} \tau_R(\omega)<\infty\right\}$. Then

$$
\mathscr{A}=\bigcup_{N>0}\left\{\sup _R \tau_R(\omega) \leq N\right\}=\bigcup_N \bigcap_R\left\{\left|\tau_R(\omega)\right| \leq N\right\},
$$

and

$$
\mathbb{P}\left(\left|\tau_R(\omega)\right| \leq N\right)=\mathbb{P}\left(\sup _{t \in[0, N]}\left\|\omega_t^R\right\|_{h-1,2}>R\right).
$$
In order to finish the proof of global existence we use the analogous of Lemma 24 in \cite{crisan2020wellposedness} and the fact that
$$
\mathbb{P}\left(\left\|\omega_t^R\right\|_{k-1,2}>R\right) \leq \frac{\mathbb{E}\left[\ln \left(\left\|\omega_t^R\right\|_{k-1,2}^2+e\right)\right]}{R^2+e} \leq \frac{\mathcal{C}\left(\omega_0, T\right)}{R^2+e} \xrightarrow[R \rightarrow \infty]{ } 0.
$$
It follows that
$$
\mathbb{P}\left(\bigcap_R\left\{\left|\tau_R(\omega)\right| \leq N\right\}\right)=\lim _{R \rightarrow \infty} \mathbb{P}\left(\left|\tau_R(\omega)\right| \leq N\right)=0
$$
and therefore $\mathbb{P}(\mathscr{A})=0$. This concludes the global existence for the solution of the equation \eqref{eq:main_equation}.
\end{proof}

\subsection{The limit of $(g^{N, +}, g^{N, -})$ coincides with the solution to the 2d-SNN}

Let $(\omega^+, \omega^-)$ be the unique solution in the sense of Definition \ref{def:solution_of_system}, and let $\xi$ be an accumulation point of $g^N$, in other words let $(g^{N, +}, g^{N, -})$ converge to $(\xi^+, \xi^-)$ in $C([0, T], \HH^{\eta}_p(\T^2))^2$ in probability. Our goal is to prove the following result:
\begin{prop}
\label{prop:uniqueness}
The following holds:
\[
\prob \left[ \norm{(\omega^+, \omega^-) - (\xi^+, \xi^-)}_{\HH^{-1+\epsilon}_2(\T^2)^2} = 0, \forall t \in [0, T]\right] = 1
\]
\end{prop}

\begin{proof}
We follow a standard argument. Let:
\[
\left\{e_k | e_k(x) := \frac{1}{(2 \cdot \pi)^{\frac{3}{2}}} e^{i <k, x>}, m \in \mathbb{Z}^2, x \in [-\pi, \pi]^2\right\},
\]
be an Hilbert basis for $\HH^{-1+\epsilon}_2(\T^2)$. Let us prove the Proposition for $\xi^+$. The case $g^{N, -}$ is analogous.

Now:
\[
\norm{\xi^+_t - \omega^+_t}_{\HH^{-1+\epsilon}_2(\T^2)}^2 = \sum\limits_{m \in \Z} \left(1+|m|^2\right)^{(-2+2\epsilon)/2}\left|\left\langle \xi^+_t - \omega^+_t , e_m\right\rangle\right|^2.
\]
Let us use It\^{o} formula to obtain:
\[
\left|\left\langle \xi^+_t - \omega^+_t , e_m\right\rangle\right|^2 =\sum_{j=1}^4 \int_0^t J_j^m(s) \mathrm{d} s+\sum_{j=5}^6 \int_0^t J_j^m(s) \mathrm{d} W_s^k,
\]
where:
\[
\begin{aligned}
J_1^m(s) &= \left\langle \xi^+_s - \omega^+_s, e_{-m}\right\rangle\left\langle \xi^+_s -\omega^+_s, \Delta e_m\right\rangle \\
&- \left\langle \xi^+_s - \omega^+_s, e_{-m}\right\rangle\left\langle \left( F(K \star \xi^+_s) - F(K \star \omega^+_s) \right) \xi^+_s, \nabla e_m\right\rangle \\
&- \left\langle \xi^+_s - \omega^+_s, e_{-m}\right\rangle\left\langle F \left( K \star \omega^+_s\right) \left( \xi^+_s - \omega^+_s \right), \nabla e_m\right\rangle,
\end{aligned}
\]
and
\[
\begin{aligned}
& J_2^m(s)=\sum_{k=1}^d\left|\left\langle\sigma_k \cdot \nabla \left(\xi^+_s - \omega^+_s\right), e_m\right\rangle\right|^2, \\
& J_3^m(s)=\frac{1}{2} \sum_{k=1}^d\left\langle \xi^+_s - \omega^+_s, e_{-m}\right\rangle\left\langle\sigma_k \cdot \nabla\left(\sigma_k \cdot \nabla \left( \xi^+_s - \omega^+_s \right)\right), e_m\right\rangle, \\
& J_4^m(s)=-\sum_{k=1}^d\left\langle \xi^+_s - \omega^+_s, e_{-m}\right\rangle\left\langle\sigma_k \cdot \nabla \left( \xi^+_s - \omega^+_s \right), e_m\right\rangle. \\
\end{aligned}
\]

Let us start with $J_1^m(s)$. We have:
\[
\begin{aligned}
& \sum_{m \in \mathbb{Z}^2} \frac{1}{\left(1+|m|^2\right)^{1-\varepsilon}} J_1^m(s) \\
& =\sum_{m \in \mathbb{Z}^2} \frac{1}{\left(1+|m|^2\right)^{1-\varepsilon}}\left\langle \xi^+_s - \omega^+_s, e_{-m}\right\rangle\left\langle \xi^+_s - \omega^+_s, \Delta e_m\right\rangle- \\
& \sum_{m \in \mathbb{Z}^2} \frac{1}{\left(1+|m|^2\right)^{1-\varepsilon}} \left\langle \xi^+_s - \omega^+_s, e_{-m}\right\rangle \left( \left\langle \left( F(K \star \xi^+_s) - F(K \star \omega^+_s) \right) \omega^+_s, \nabla e_m\right\rangle - \left\langle F\left( K \star \omega^+_s\right) \left( \xi^+_s - \omega^+_s\right), \nabla e_m\right\rangle \right) \\
& =-\sum_{m \in \mathbb{Z}^2} \frac{|m|^2}{\left(1+|m|^2\right)^{1-\varepsilon}}\left|\left\langle \xi^+_s - \omega^+_s, e_m\right\rangle\right|^2 \\
&- \sum_{m \in \mathbb{Z}^2} \frac{i m}{\left(1+|m|^2\right)^{1-\varepsilon}}\left\langle \xi^+_s - \omega^+_s, e_{-m}\right\rangle \left( \left\langle \left( F(K \star \omega^+_s) - F(K \star \xi^+_s) \right) \omega^+_s,  e_m\right\rangle - \left\langle F\right( K \star \xi^+_s \left) \left( \xi^+_s - \omega^+_s \right), e_m\right\rangle \right).
\end{aligned}
\]
Therefore one has:
\[
\begin{aligned}
& \sum_{m \in \mathbb{Z}^2} \frac{1}{\left(1+|m|^2\right)^{1-\varepsilon}} J_1^m(s) \leq  -\sum_{m \in \mathbb{Z}^2} \frac{|m|^2}{\left(1+|m|^2\right)^{1-\varepsilon}}\left|\left\langle \xi^+_s - \omega^+_s, e_m\right\rangle\right|^2 \\
&+ \sum_{m \in \mathbb{Z}^2} \frac{1}{2\left(1+|m|^2\right)^{1-\varepsilon}}\bigg(|m|^2\left|\left\langle \xi^+_s - \omega^+_s, e_{-m}\right\rangle\right|^2+  \left| \left\langle \left( F(K \star \xi^+_s ) - F(K \star \omega^+_s) \right) \omega^+_s, e_m\right\rangle \right|^2 \\
&+ \left| \left\langle F \left( K \star \omega^+_s\right) \left(\xi^+_s - \omega^+_s \right), e_m\right\rangle \right|^2 \bigg) \\
\leqslant & -\frac{1}{2}\left\|\xi^+_s - \omega^+_s\right\|_{\HH^{\epsilon}_2(\T^2)}^2+\frac{1}{2}\left\|\left( F(K \star \xi^+_s) - F(K \star \omega^+_s) \right) \omega^+_s\right\|_{\HH^{-1+\epsilon}_2(\T^2)}^2 +\frac{1}{2}\left\|F\left(K \star \xi^+_s \right) \left( \xi^+_s - \omega^+_s\right)\right\|_{\HH^{-1+\epsilon}_2(\T^2)}^2 \\
\leqslant & -\frac{1}{2}\left\|\xi^+_s - \omega^+_s\right\|_{\HH^{\epsilon}_2(\T^2)}^2+C\left(1+\left\| \omega^+_s\right\|_{B_{\infty,2}^{-1+\varepsilon}}^2\right) \|\xi^+_s - \omega^+_s\|_{\HH^{-1+\epsilon}_2(\T^2)}^2.
\end{aligned}
\]
For the last inequality we used Corollary 2 of \cite{10.1214/16-AOP1116}. For the penultimate inequality we used:
\[
-\frac{|m|^2}{\left(1+|m|^2\right)^{1-\varepsilon}} \leq -C\left( 1+|m|^2 \right)^{\varepsilon}, \quad \forall m \in \Z^2,
\]
and Lemma \ref{lem:K_negative_epsilon} in the Appendix.

Furthermore, we have:
$$
\begin{aligned}
& \sum_{m \in \mathbb{Z}^2} \frac{1}{\left(1+|m|^2\right)^{1-\varepsilon}}\left(J_2^m(s)+J_3^m(s)\right) \\
= & \sum_{k=1}^d\left(\left\|\sigma_k \cdot \nabla \left(\xi^+_s - \omega^+_s \right)\right\|_{\HH^{-1+\varepsilon}_2(\T^2)}^2+\left\langle\sigma_k \cdot \nabla\left(\sigma_k \cdot \nabla \left( \xi^+_s - \omega^+_s \right)\right), \xi^+_s - \omega^+_s\right\rangle_{\HH^{-1+\varepsilon}_2(\T^2)}\right) \\
\leqslant & C\left\|\xi^+_s - \omega^+_s\right\|_{\HH^{-1+\varepsilon}_2(\T^2)}^2
\end{aligned}
$$
where $C$ depends on $\left\{\sigma_k\right\}_{k \in \N}$. The last inequality follows from Lemma 2.3 in \cite{krylov2015hypoellipticity}.
By Cauchy-Schwarz inequality, we have
$$
\begin{aligned}
& \int_0^T\left(\sum_{m \in \mathbb{Z}^2} \frac{1}{\left(1+|m|^2\right)^{1-\varepsilon}} J_5^m(s)\right)^2 \mathrm{~d} s \\
= & \int_0^T\left(\sum_{k=1}^d \sum_{m \in \mathbb{Z}^2} \frac{1}{\left(1+|m|^2\right)^{1-\varepsilon}}\left\langle \xi^+_s - \omega^+_s, e_{-m}\right\rangle\left\langle\sigma_k \cdot \nabla \left( \xi^+_s - \omega^+_s \right), e_m\right\rangle\right)^2 \mathrm{~d} s \\
\lesssim & \int_0^T\left(\sum_{k=1}^d \sum_{m \in \mathbb{Z}^2} \frac{1}{\left(1+|m|^2\right)^{1-\varepsilon}}\left\langle \xi^+_s - \omega^+_s, e_{-m}\right\rangle^2\right)\left(\sum_{k=1}^d \sum_{m \in \mathbb{Z}^2} \frac{1}{\left(1+|m|^2\right)^{1-\varepsilon}}\left\langle\sigma_k \cdot \nabla \left( \xi^+_s - \omega^+_s \right), e_m\right\rangle^2\right) \mathrm{d} s \\
\lesssim & \int_0^T\left\|\xi^+_s - \omega^+_s\right\|_{\HH^{-1+\varepsilon}}^2\left\|\sigma_k \cdot \nabla \left( \xi^+_s - \omega^+_s\right)\right\|_{\HH^{-1+\varepsilon}_2(\T^2)}^2 \mathrm{~d} s \lesssim \int_0^T\left\|\xi^+_s - \omega^+_s\right\|_{\HH^{-1+\varepsilon}_2(\T^2)}^4\mathrm{~d} s,
\end{aligned}
$$
again by Lemma 2.3 in \cite{krylov2015hypoellipticity}, and we deduce that

\[
\E \left[ \int_0^T\left(\sum_{m \in \mathbb{Z}^2} \frac{1}{\left(1+|m|^2\right)^{1-\varepsilon}} J_5^m(s)\right)^2 \mathrm{~d} s  \right] \leq \E \left[ \int_0^T\left\|\xi^+_s - \omega^+_s\right\|_{\HH^{-1+\varepsilon}_2(\T^2)}^4\mathrm{~d} s \right].
\]
This allows us to conclude that:
\[
M_t:=\int_0^t \sum_{m \in \mathbb{Z}^2} \frac{1}{\left(1+|m|^2\right)^{2+\varepsilon}}J_5^m(s) \mathrm{d} W_s^k, 
\ \ t\ge 0,\]
is a continuous martingale.
Finally, by merging the previous inequalities, we have:
$$
\left\|\xi^+_t - \omega^+_t\right\|_{\HH^{-1+\varepsilon}_2(\T^2)}^2 \lesssim \int_0^t\left(1+\left\|\omega^+_s\right\|_{B_{\infty, 2}^{-1+\varepsilon}}^2\right)\left\|\xi^+_s - \omega^+_s\right\|_{\HH^{-1+\varepsilon}_2(\T^2)}^2 \mathrm{~d} s+M_t .
$$

By applying the stochastic Gronwall's inequality (Corollary 5.4 in \cite{Gei21}) we conclude:
\[
\mathbb{P}\left(\left\|\xi^+_t - \omega^+_t\right\|_{\HH^{-1+\varepsilon}_2(\T^2)}=0, \forall t \in[0, T]\right)=1.
\]

\end{proof}

%%%%%%%%%%%%%%%%%%%%%%%%%%%%%%%%%%%%%%%%%%%%%%%%%%%%%

%%%%%%%%%%%%%%%%%%%%%%%%%%%%%%%%%%%%%%%%%%%%%%%%%%%%%
\section{Proof of the main result}
\label{sec:convergence-probability}
To prove the Theorem \ref{th:main_theorem}, we make use of the following Gyongy-Krylov Lemma (see  \cite{gyongy} for a proof):

\begin{lemma}
\label{lemma:gyongy}
Let $\left\{X_n\right\}$ be a sequence of random elements in a Polish space $E$ equipped with the Borel $\sigma$-algebra. Then $X_n$ converges in probability to a $E$-valued random element if and only if for each pair $\left(X_{\ell}, X_m\right)$ of subsequences, there exists a subsequence $\left\{v_k\right\}$ given by
\[
v_k=\left(X_{\ell(k)}, X_{m(k)}\right)
\]
converging weakly to a random element $v(x, y)$ supported on the diagonal set
\[
\{(x, y) \in E \times E: x=y\}.
\]
\end{lemma}

\begin{proof}(of Theorem \ref{th:main_theorem}])
Let us consider a pair $v_{i, j} = (g^{i, +}, g^{i, -}, g^{j, +}, g^{j, -})_{i, j}$ with values in $\mathfrak{Y}^2 \times \mathfrak{Y}^2$. Because of Section \ref{sec:compactness}, we can conclude that there exists a subsequence, say $v_{i_k, j_k}$, such that $v_{i_k, j_k}$ converges weakly to a probability $\nu$ on $\mathfrak{Y}^2 \times \mathfrak{Y}^2$.

Now, because of Skorokhod representation theorem (see p.70 in \cite{skorohod}), modulo changing probability space, a random variable $(\tilde{u}^+, \tilde{u}^-, \overline{u}^+, \overline{u}^-)$ exists in the new probability space, and $(\overline{g}^{i_k, +}, \overline{g}^{i_k, -}, \overline{g}^{j_k, +}, \overline{g}^{j_k, -})_{i, j}$ such that:
\[
(\overline{g}^{i_k, +}, \overline{g}^{i_k, -}, \overline{g}^{j_k, +}, \overline{g}^{j_k, -})_{i, j} \xrightarrow[]{\text{weakly}} (\tilde{u}^+, \tilde{u}^-, \overline{u}^+, \overline{u}^-),
\]
with $\text{Law}\left(\overline{g}^{i_k, +}, \overline{g}^{i_k, -}, \overline{g}^{j_k, +}, \overline{g}^{j_k, -} \right) = \text{Law}\left( g^{i_k, +}, g^{i_k, -}, g^{j_k, +}, g^{j_k, -} \right)$. Therefore, by Section \ref{sec:passing-to-the-limit} one concludes that $(\tilde{u}^+, \tilde{u}^-)$ and $(\overline{u}^+, \overline{u}^-)$ are both solutions in the sense of Definition \ref{def:solution_of_system}. Now, because of Proposition \ref{prop:uniqueness}, one concludes $(\tilde{u}^+, \tilde{u}^-) = (\overline{u}^+, \overline{u}^-)$ a.s.

Now, we notice that this implies:
\[
\nu \left( \left\{(x, y) \in \mathfrak{Y}^2 \times \mathfrak{Y}^2: x=y \right\} \right) = 0,
\]
and therefore we can finally apply Lemma \ref{lemma:gyongy} to get that the original sequence in the original probability space converges in probability.
\end{proof}

%%%%%%%%%%%%%%%%%%%%%%%%%%%%%%%%%%%%%%%%%%%%%%%%%%%%%
\section{Technical estimates}
Let us define:
\[
A:\HH^2_p(\T^2)\subset \LL^p(\T^2) \xrightarrow[]{} \LL^p(\T^2),
\]
as:
\[
Af(x):= \nu \Delta f(x) + \sum\limits_k \left( \sigma^k (x)\right)^T \left( \nabla^2 f (x) \right)  \sigma^k(x).
\]

Notice that:
\[
\| (I-A)^{\frac{\alpha}{2}} f\|_{\LL^p(\T^2)} \simeq \| f\|_{\HH ^{\alpha}_p(\T^2)},
\]
and that we are able to express the equation satisfied by $g^{N, +}$ in mild form (and similarly for $g^{N, -}_t$). In other words, the following holds:
\begin{equation}
\label{eq:mild_form}
\begin{aligned}
g_t^{N, +} = & e^{tA}g_0^{N, +} + \int_0^t e^{(t-s)A} \left( \nabla V^N \star \left( F\left( K \star g_s^{N} \right) \mu_s^{N, +} \right) \right) ds + \\
&+ \frac{\sqrt{2\nu}}{N}\sum\limits_{i=1}^N \int_0^t e^{(t-s)A} \left( \nabla V^N(\cdot-X^{i,N, +}_t) \right) d B^{i, +}_s + \\
&+ \sum\limits_{k \in \mathbb{N}} \int_0^t e^{(t-s)A} \left( \nabla V^N \star \left( \sigma_k \mu_s^{N, +} \right) \right) dW^k_s + \\
&+ \sum\limits_{k \in \mathbb{N}} \int_0^t e^{(t-s)A} \left\langle \mu_s^{N, +}, \left(\nabla \sigma_k \cdot \sigma_k \right) \cdot \nabla V^N(x-\cdot) \right\rangle ds.
\end{aligned}
\end{equation}

\begin{lemma}
\label{lemma}
Let $q \geq 2$. Then there exists $C$, independent of $N$ and $t \in [0,T]$, such that:
\begin{equation}
    \label{eq:bound_stochastic_term_1}
    \E \left[ \bigg\|  \frac{1}{N}\sum\limits_{i=1}^N \int_0^t  (I-A)^{\frac{\alpha}{2}} e^{(t-s)A} \nabla (V^N(\cdot - X^{i,N, \pm}_s)) d B^{i, \pm}_s \bigg\|_{\LL^p}^q \right] \leq C
\end{equation}
and:
\begin{equation}
    \label{eq:bound_stochastic_term_2}
    \E \left[ \bigg\| \int_0^t (I-A)^{\frac{\alpha}{2}} e^{(t-s)A} \left( \nabla V^N * \left( \sigma_k \mu_s^{N, \pm} \right) \right) dW^k_s \bigg\|_{\LL^p}^q \right] \leq C.
\end{equation}
\end{lemma}

\begin{proof}
It is sufficient to prove \eqref{eq:bound_stochastic_term_1} because \eqref{eq:bound_stochastic_term_2} is similar. One has:
\[
\int_0^t (I-A)^{\alpha / 2} e^{(t-s)A} \left( \nabla V^N \star \left( \sigma_k \mu_s^{N, \pm} \right) \right) dW^k_s = \frac{1}{N} \sum_{i=1}^N \int_0^t(\mathrm{I}-A)^{\alpha / 2} e^{(t-s) A} \nabla_x \left( \sigma^k (X_s^{i, N, \pm}) V^N\left(x-X_s^{i, N, \pm}\right)\right) dW^k_s
\]

Furthermore, it is sufficient to prove it with $\mu_s^{N,+}$, and the corresponding result with $\mu_s^{N,-}$ is the same. For the sake of simplicity, we omit the sign.
From Sobolev embeddings we have
$$
\begin{aligned}
& \E \left[ \left\|\frac{1}{N} \sum_{i=1}^N \int_0^t(\mathrm{I}-A)^{\alpha / 2} e^{(t-s) A} \nabla_x \left(V^N\left(x-X_s^{i, N}\right)\right) d B_s^i\right\|_{\LL^p}^q \right] \\
& \quad \leqslant C \cdot \E \left[ \left\|\frac{1}{N} \sum_{i=1}^N \int_0^t(\mathrm{I}-A)^{\left(1+\alpha-\frac{2}{p}\right) / 2} e^{(t-s) A} \nabla_x \left(V^N\left(x-X_s^{i, N}\right)\right) d B_s^i\right\|_{\LL^2}^q \right].
\end{aligned}
$$

From the Burkholder-Davis-Gundy inequality (see \cite{bdg} for instance) we obtain
$$
\begin{aligned}
& \E \left[ \left\|\frac{1}{N} \sum_{i=1}^N \int_0^t(\mathrm{I}-A)^{\left(1+\alpha-\frac{2}{p}\right) / 2} e^{(t-s) A}\nabla_x \left(V^N\left(x-X_s^{i, N}\right)\right) d B_s^i\right\|_{\LL^2}^q \right] \\
& \quad \leq C_q \mathbb{E}\left[ \left( \frac{1}{N^2} \sum_{i=1}^N \int_0^t\left\|(\mathrm{I}-A)^{\left(1+\alpha-\frac{2}{p}\right) / 2} e^{(t-s) A} \nabla_x \left(V^N\left(x-X_s^{i, N}\right)\right)\right\|_{\mathbb{L}^2\left(\T^2\right)}^2 d s \right)^{q / 2} \right].
\end{aligned}
$$

Moreover, we can estimate
$$
\frac{1}{N^2} \int_{\mathbb{T}^2} \sum_{i=1}^N \int_0^t\left|\left((\mathrm{I}-A)^{\left(1+\alpha-\frac{2}{p}\right) / 2} e^{(t-s) A} \nabla_x \left(V^N\left(x-X_s^{i, N}\right)\right)\right)(x)\right|^2 d s d x
$$
Finally, we obtain that:
\begin{equation}
    \begin{aligned}
        \frac{1}{N^2}  \sum\limits_{i=1}^N \int_0^t \int_{\T^2}  \bigg| \bigg( (I-A)^{(1+\alpha-\frac{2}{p})/2} e^{(t-s)A} \nabla_x (V^N(x - X^{i,N}_s)) \bigg) (x) \bigg|^2 ds dx = \\
        = \frac{1}{N} \int_0^t \| (I-A)^{(1+\alpha-\frac{2}{p})/2} e^{(t-s)A} \nabla V^N\|_{\LL^2(\R^2)}^2 ds \\
        = \frac{1}{N} \int_0^t \| (I-A)^{\frac{1}{2}-\frac{\delta}{2}} e^{(t-s)A} (I-A)^{(1+\alpha-\frac{2}{p}-(1-\delta))/2}\nabla V^N\|_{\LL^2(\R^2)}^2 \leq \\
        \leq \frac{1}{N} \int_0^t \| (I-A)^{\frac{1}{2}-\frac{\delta}{2}} e^{(t-s)A} \|_{\LL^2 \to \LL^2}^2 \|(I-A)^{(\alpha-\frac{2}{p}+\delta)/2}\nabla V^N\|_{\LL^2(\R^2)}^2 ds \leq \\
        \leq \frac{1}{N} \int_0^t \frac{1}{(t-s)^{1-\delta}} \| (I-A)^{(\alpha-\frac{2}{p}+\delta)/2} \nabla V^N\|_{\LL^2}^2 ds
    \end{aligned}
\end{equation}
Let us now study the term $\| (I-A)^{(\alpha-\frac{2}{p}+\delta)/2} \nabla V^N\|_{\LL^2}^2$. By Theorem 6.10 (p. 73) of \cite{pazy} we have that by defining $\gamma= \frac{\alpha}{2}-\frac{1}{p}+\frac{\delta}{2}$, we get
\begin{equation}
    \begin{aligned}
        \| (I-A)^{(\alpha-\frac{2}{p}+\delta)/2} \nabla V^N\|_{\LL^2}^2 & = \| (I-A)^{\gamma} \nabla V^N\|_{\LL^2}^2 \\
        & \underset{\text{Theorem }6.10 \text{ of \cite{pazy} }}{\leq} (C \| \nabla V^N\|_{\LL^2}^{1-\gamma} \cdot \|(I-A) \nabla V^N\|_{\LL^2}^\gamma)^2
        \end{aligned}
        \end{equation}
Now let us study separately the two terms:
\begin{equation}
\label{eq:mollifier_estimate_1}
    \left( \left\| \nabla V^N \right\|_{\LL^2}^{1-\gamma} \right) ^2
\end{equation}
and
\begin{equation}
\label{eq:mollifier_estimate_2}
    \left( \left\|(I-A) \nabla V^N \right\|_{\LL^2}^\gamma \right)^2
\end{equation}
For the first term \eqref{eq:mollifier_estimate_1}:
\[
\begin{aligned}
 & \left( \left\| \nabla V^N \right\|_{\LL^2}^{1-\gamma} \right)^2 = \left( \int_{\T^2} \left| \nabla V^N(x) \right|^2 dx \right)^{1-\gamma} \\
 & = N^{4\beta(1-\gamma)}\left( \int_{\T^2} \left| N^{\beta} \nabla V(N^{\beta}x) \right|^2 dx \right)^{1-\gamma}
 \end{aligned}
\]
and, with a change of variables, we have that:
\[
\left( \int_{\T^2} \left| N^{\beta} \nabla V(N^{\beta}x) \right|^2 dx \right)^{1-\gamma} = \left( \int_{\T^2} \left| \nabla V(y) \right|^2 dy \right)^{1-\gamma} < \infty,
\]
that does not depend on $N$.
Let us study now \eqref{eq:mollifier_estimate_2}.
We have:
\[
\begin{aligned}
\left( \left\| \left( I - A \right)\nabla V^N \right\|_{\LL^2}^{\gamma} \right)^2 & = \left( \int_{\T^2} \left| \left( I - A \right)\nabla V^N(x) \right|^2 dx \right)^{\gamma} \\
 & = N^{4\beta \gamma}\left( \int_{\T^2} \left| N^{3\beta} \left( I - A \right) \nabla V(N^{\beta}x) \right|^2 dx \right)^{\gamma} \\
 & =  N^{4\beta \gamma} N^{4\beta \gamma}\left( \int_{\T^2} \left| N^{\beta} \left( I - A \right) \nabla V(N^{\beta}x) \right|^2 dx \right)^{\gamma} \\
 & = N^{8\beta \gamma}\left( \int_{\T^2} \left| N^{\beta} \left( I - A \right) \nabla V(N^{\beta}x) \right|^2 dx \right)^{\gamma}
\end{aligned}
\]
Therefore:
\[
\| (I-A)^{(\alpha-\frac{2}{p}+\delta)/2} \nabla V^N\|_{\LL^2}^2 \lesssim N^{4 \beta (1-\gamma)} \cdot N^{8 \beta \gamma}=N^{4\beta (1 + \gamma)} = N^{4 \beta (\frac{\alpha}{2} - \frac{1}{p} + 1 + \frac{\delta}{2})}.
\]
Now from Assumption \ref{ass:assprinci} we have:
\begin{itemize}
    \item $\frac{2}{p} < \alpha < 1$ that can be rewritten as $0 < \frac{\alpha}{2}- \frac{1}{p}< 1- \frac{1}{p}$,
    \item $0 < \beta < \frac{1}{4+2\alpha-\frac 4 p} < \frac 1 4$ that can be rewritten as $4 \beta (\frac{\alpha}{2} - \frac{1}{p} + 1) < 1$.
\end{itemize}
It is obvious that, if $4 \beta (\frac{\alpha}{2} - \frac{1}{p} + 1) < 1$, then there exists a $\delta > 0$ such that also $4 \beta (\frac{\alpha}{2} - \frac{1}{p} + 1 + \frac{\delta}{2}) < 1$ holds.
Thus, with this choice of $\delta$, dividing by $N$ we get:
$$
\frac{1}{N^2} \int_{\mathbb{T}^2} \sum_{i=1}^N \int_0^t\left|\left((\mathrm{I}-A)^{\left(1+\alpha-\frac{2}{p}\right) / 2} e^{(t-s) A} \nabla_x \left(V^N\left(x-X_s^{i, N}\right)\right)\right)(x)\right|^2 d s d x \xrightarrow[N \to \infty]{} 0.
$$
\end{proof}

From now on, to simplify notation, denote $\tilde{A} := \left( I - A \right)$

\begin{prop}
    \label{prop1}
    Let $q \geq 2$. Then there exists a positive constant $C$ which does not depend on $N$ nor $t \in (0,T]$ such that:
    \[
    \E \bigg[ \bigg\| \tilde{A}^{\frac{\alpha}{2}} g_t^{N, \pm} \bigg\|_{\LL^p(\T^2)}^q \bigg] \leq C.
    \]
\end{prop}

\begin{proof}
In the following, for the sake of simplicity, when the sign is omitted we mean it to be $+$. We will work with $g^{N, +}$ because the proof for $g^{N, -}$ is the same. Let us consider the mild form \eqref{eq:mild_form} for the process $g^N$ and let us apply $(I-A)^{\frac{\alpha}{2}}$ to both sides of the equality. The goal is to prove the following estimate:
\[
\E \left[ \left\| \tilde{A}^{\frac{\alpha}{2}} g_t^N \right\|_{\LL^p}^q \right] \leq C + \int_0^t \frac{C_{\nu, T}}{(t-s)^{(1+\alpha) / 2}} \E \left[ \left\| \tilde{A}^{\frac{\alpha}{2}} g_s^N \right\|_{\LL^p}^q \right] ds
\]
Now, we need to estimate:
\begin{equation}
\label{eq:mild_term_1}
\int_0^t \tilde{A}^{\frac{\alpha}{2}}e^{(t-s)A} \left( \nabla V^N \star \left( F\left(K \star g_s^{N} \right) \mu_s^N \right) \right) ds,
\end{equation}
and
\begin{equation}
\label{eq:mild_term_2}
    \begin{aligned}
        \sum\limits_{k \in \mathbb{N}} \int_0^t \frac{1}{N}\sum\limits_{i=1}^N \tilde{A}^{\frac{\alpha}{2}} e^{(t-s)A} \left( \sigma_k(X^{i,N}_s) \cdot \nabla \sigma_k(X^{i,N}_s) \right) \cdot \nabla_x V^N(x-X^{i,N}_s) ds,
    \end{aligned}
\end{equation}
as the other terms are straightforward to bound (see e.g. \cite{paper}).

Let us start estimating \eqref{eq:mild_term_2}:
    \begin{equation}
    \begin{aligned}
    & \sum\limits_{k \in \mathbb{N}} \int_0^t \frac{1}{N}\sum\limits_{i=1}^N \tilde{A}^{\frac{\alpha}{2}} e^{(t-s)A} \left( \sigma_k(X^{i,N}_s) \cdot \nabla \sigma_k(X^{i,N}_s) \right) \cdot \nabla_x V^N(x-X^{i,N}_s) ds \\
    &=  \sum\limits_{k \in \mathbb{N}} \int_0^t \frac{1}{N}\sum\limits_{i=1}^N \tilde{A}^{\frac{\alpha}{2}} e^{(t-s)A} \nabla_x \left( \sigma_k(X^{i,N}_s) \cdot \nabla \sigma_k(X^{i,N}_s) V^N(x-X^{i,N}_s) \right) ds \\
    &=  \sum\limits_{k \in \mathbb{N}} \int_0^t \frac{1}{N}\sum\limits_{i=1}^N \tilde{A}^{\frac{\alpha}{2}+\frac{1}{2}} e^{(t-s)A} \tilde{A}^{-\frac{1}{2}}\nabla_x \left( \sigma_k(X^{i,N}_s) \cdot \nabla \sigma_k(X^{i,N}_s) V^N(x-X^{i,N}_s) \right) ds \\
    &=  \sum\limits_{k \in \mathbb{N}} \int_0^t \tilde{A}^{\frac{\alpha}{2}+\frac{1}{2}} e^{(t-s)A} \tilde{A}^{-\frac{1}{2}} \nabla_x \frac{1}{N}\sum\limits_{i=1}^N \left( \sigma_k(X^{i,N}_s) \cdot \nabla \sigma_k(X^{i,N}_s) V^N(x-X^{i,N}_s) \right) ds
    \end{aligned}
    \end{equation}

We already estimated the other two in \ref{lemma}.
Now, we notice that, defining $H:= \LL^p(\T^2)$:
\begin{equation}
\begin{aligned}
& \int_0^t\left\|\tilde{A}^{\frac{\alpha}{2}+\frac{1}{2}} e^{(t-s)A} \tilde{A}^{-\frac{1}{2}}\nabla_x \cdot \frac{1}{N}\sum\limits_{i=1}^N \left( \sigma_k(X^{i,N}_s) \cdot \nabla \sigma_k(X^{i,N}_s) \right) \cdot V^N(x-X^{i,N}_s)\right\|_{\mathbb{L}^q(\Xi, H)} d s \\
&= \int_0^t \mathbb{E} \left[  \left\| \tilde{A}^{\frac{\alpha}{2}+\frac{1}{2}} e^{(t-s)A} \tilde{A}^{-\frac{1}{2}}\nabla_x \cdot\frac{1}{N}\sum\limits_{i=1}^N \left( \sigma_k(X^{i,N}_s) \cdot \nabla \sigma_k(X^{i,N}_s) \right) \cdot V^N(x-X^{i,N}_s)\right\|_{\LL^p}^q \right]^{1/q} d s \\
&\leq \int_0^t  \mathbb{E} \left[ \left\|\tilde{A}^{(1+\alpha) / 2} e^{(t-s) A}\right\|_{\mathbb{L}^p \rightarrow \mathbb{L}^p}^q \left\| \frac{1}{N}\sum\limits_{i=1}^N \left( \sigma_k(X^{i,N}_s) \cdot \nabla \sigma_k(X^{i,N}_s) \right) \cdot V^N(x-X^{i,N}_s)\right\|_{\LL^p}^q \right]^{1/q} d s \\
&\leq \int_0^t \left\|\tilde{A}^{(1+\alpha) / 2} e^{(t-s) A}\right\|_{\mathbb{L}^p \rightarrow \mathbb{L}^p} \mathbb{E} \left[ \left\| \frac{1}{N}\sum\limits_{i=1}^N \left( \sigma_k(X^{i,N}_s) \cdot \nabla \sigma_k(X^{i,N}_s) \right) \cdot V^N(x-X^{i,N}_s)\right\|_{\LL^p}^q \right]^{1/q} d s \\
&\leq \left\| \sigma_k \right\|_{\LL^{\infty}} \left\| \nabla \sigma_k \right\|_{\LL^{\infty}} \int_0^t \left\|\tilde{A}^{(1+\alpha) / 2} e^{(t-s) A}\right\|_{\mathbb{L}^p \rightarrow \mathbb{L}^p} \mathbb{E} \left[ \left\| \frac{1}{N}\sum\limits_{i=1}^N V^N(x-X^{i,N}_s)\right\|_{\LL^p}^{q} \right]^{1/q} d s \\
&= C_k \int_0^t \left\|\tilde{A}^{(1+\alpha) / 2} e^{(t-s) A}\right\|_{\mathbb{L}^p \rightarrow \mathbb{L}^p} \E \left[ \left\| g_s^N \right\|_{\LL^p}^q \right]^{1/q} d s \leq \\
& \leq C_k \int_0^t \frac{C_{\nu, T}}{(t-s)^{(1+\alpha) / 2}} \E \left[ \left\| g_s^N \right\|_{\LL^p}^q \right]^{1/q} d s.
\end{aligned}
\end{equation}
For the penultimate identity, we used that 
\begin{equation}
\begin{aligned}
    &\left\| \frac{1}{N}\sum\limits_{i=1}^N \left( \sigma_k(X^{i,N}_s) \cdot \nabla \sigma_k(X^{i,N}_s) \right) \cdot V^N(x-X^{i,N}_s)\right\|_{\LL^p}^p \\
    &= \int_{\T^2} \left| \frac{1}{N}\sum\limits_{i=1}^N \left( \sigma_k(X^{i,N}_s) \cdot \nabla \sigma_k(X^{i,N}_s) \right) \cdot V^N(x-X^{i,N}_s)\right|^p dx \\
    &= \int_{\T^2} \left| \frac{1}{N}\sum\limits_{i=1}^N \left( \sigma_k(X^{i,N}_s) \cdot \nabla \sigma_k(X^{i,N}_s) \right) \cdot V^N(x-X^{i,N}_s)\right|^p dx \\
    &\leq \int_{\T^2} \left( \frac{1}{N}\sum\limits_{i=1}^N \left| \left( \sigma_k(X^{i,N}_s) \cdot \nabla \sigma_k(X^{i,N}_s) \right) \cdot V^N(x-X^{i,N}_s) \right| \right)^p dx \\
    &\leq \int_{\T^2} \left( \frac{1}{N}\sum\limits_{i=1}^N \left| \left( \sigma_k(X^{i,N}_s) \cdot \nabla \sigma_k(X^{i,N}_s) \right) \cdot V^N(x-X^{i,N}_s) \right| \right)^p dx \\
    &\leq \left\| \sigma_k \right\|_{\LL^{\infty}}^p \left\| \nabla \sigma_k \right\|_{\LL^{\infty}}^p \int_{\T^2} \left( \frac{1}{N}\sum\limits_{i=1}^N \left| V^N(x-X^{i,N}_s) \right| \right)^p dx \\
    &=\left\| \sigma_k \right\|_{\LL^{\infty}}^p \left\| \nabla \sigma_k \right\|_{\LL^{\infty}}^p \int_{\T^2} \left( \frac{1}{N}\sum\limits_{i=1}^N V^N(x-X^{i,N}_s) \right)^p dx \\
    &= \left\| \sigma_k \right\|_{\LL^{\infty}}^p \left\| \nabla \sigma_k \right\|_{\LL^{\infty}}^p \int_{\T^2} \left( \frac{1}{N}\sum\limits_{i=1}^N V^N(x-X^{i,N}_s) \right)^p dx \\
    &= \left\| \sigma_k \right\|_{\LL^{\infty}}^p \left\| \nabla \sigma_k \right\|_{\LL^{\infty}}^p \left\| V^N \star \mu^N_s \right\|_{\LL^p(\T^2)}^p. \\
\end{aligned}
\end{equation}
Therefore:
\[
\left\| \frac{1}{N}\sum\limits_{i=1}^N \left( \sigma_k(X^{i,N}_s) \cdot \nabla \sigma_k(X^{i,N}_s) \right) \cdot V^N(x-X^{i,N}_s)\right\|_{\LL^p} \leq \left\| \sigma_k \right\|_{\LL^{\infty}} \left\| \nabla \sigma_k \right\|_{\LL^{\infty}} \left\| V^N \star \mu^N_s \right\|_{\LL^p(\T^2)}
\]
and thus:
\[
\left\| \frac{1}{N}\sum\limits_{i=1}^N \left( \sigma_k(X^{i,N}_s) \cdot \nabla \sigma_k(X^{i,N}_s) \right) \cdot V^N(x-X^{i,N}_s)\right\|_{\LL^p}^q \leq \left\| \sigma_k \right\|_{\LL^{\infty}}^q \left\| \nabla \sigma_k \right\|_{\LL^{\infty}}^q \left\| V^N \star \mu^N_s \right\|_{\LL^p(\T^2)}^q,
\]
and for the last inequality, we used the following classical estimate from semigroup theory (see \cite{pazy}):
$$
\left\|\tilde{A}^{(1+\alpha) / 2} e^{(t-s) A}\right\|_{\mathbb{L}^p \rightarrow \mathbb{L}^p} \leq \frac{C_{\nu, T}}{(t-s)^{(1+\alpha) / 2}}.
$$

Finally, we have that:
\begin{equation}
\begin{aligned}
    \left( \E \left[ \left\| V^N*\mu_s^N \right\|_{\LL^p}^q \right] \right)^{\frac{1}{q}} =\left( \E \left[ \left\| g_s^N \right\|_{\LL^p}^q \right] \right)^{\frac{1}{q}}\leq \left( \E \left[ \left\| \tilde{A}^{\alpha/2} g_s^N \right\|_{\LL^p}^q \right] \right)^{\frac{1}{q}},
\end{aligned}
\end{equation}
from which:
\[
\eqref{eq:mild_term_2} \leq C_k \int_0^t \frac{C_{\nu, T}}{(t-s)^{(1+\alpha) / 2}} \E \left[ \left\| g_s^N \right\|_{\LL^p}^q \right]^{1/q} d s.
\]

Let us now estimate \eqref{eq:mild_term_1}.
\begin{equation}
\begin{aligned}
& \int_0^t  \left( \E \left[ \left \| \tilde{A}^{\frac{\alpha}{2}}e^{(t-s)A} \left( \nabla V^N \star \left( F \left( K \star g_s^{N} \right) \mu_s^N \right) \right) \right\|_{\LL^p}^q \right] \right)^{\frac{1}{q}}ds = \\
& \int_0^t  \left( \E \left[ \left\| \tilde{A}^{\frac{\alpha}{2}+\frac 1 2}e^{(t-s)A} \tilde{A}^{-\frac{1}{2}} \nabla \left( V^N \star \left( F \left( K \star g_s^{N} \right) \mu_s^N \right) \right) \right\|_{\LL^p}^q \right] \right)^{\frac{1}{q}} ds = \\
& \int_0^t \E \left[ \left\| \tilde{A}^{\frac{\alpha}{2}+\frac 1 2}e^{(t-s)A} \tilde{A}^{-\frac{1}{2}} \nabla \left( V^N \star \left( F \left( K \star g_s^{N} \right) \mu_s^N \right) \right) \right\|_{\mathbb{L}^p(\T^2)}^q \right]^{\frac 1 q} ds \leq \\
& \leq \int_0^t \frac{C_{\nu, T}}{(t-s)^{(1+\alpha) / 2}} \E \left[ \left\|\tilde{A}^{-\frac{1}{2}} \nabla \left( V^N \star \left( F \left( K \star g_s^{N}\right) \mu_s^N \right) \right) \right\|_{\mathbb{L}^p(\T^2)}^q \right]^{\frac 1 q} ds \leq \\
& \leq \int_0^t \frac{C_{\nu, T}}{(t-s)^{(1+\alpha) / 2}} \E \left[ \left\| V^N \star \left( F \left( K \star g_s^{N} \right) \mu_s^N \right) \right\|_{\mathbb{L}^p(\T^2)}^q \right]^{\frac 1 q} ds.
\end{aligned}
\end{equation}
On the other hand:
\[
\left|\left(V^N *\left(F\left(K * g_s^{N, +} - K * g_s^{N, -}\right) \mu_s^N\right)\right)(x)\right| \leq\left\|F\left(K * g_s^{N, +} - K * g_s^{N, -}\right)\right\|_{\infty}\left|V^N * \mu_s^N(x)\right| \leq M\left|g_s^N(x)\right|,
\]
therefore:
\[
\begin{aligned}
\int_0^t  \left( \E \left[ \left\| \tilde{A}^{\frac{\alpha}{2}}e^{(t-s)A} \left( \nabla V^N * \left( F \left( K * g_s^{N, +} - K * g_s^{N, -} \right) \mu_s^N \right) \right) \right\|_{\LL^p} \right] \right)^{\frac{1}{q}} ds & \lesssim \int_0^t \frac{C_{\nu, T}}{(t-s)^{(1+\alpha) / 2}} \E \left[ \left\| g_s^N \right\|_{\mathbb{L}^p(\T^2)}^q \right]^{\frac 1 q} \\
& \leq \int_0^t \frac{C_{\nu, T}}{(t-s)^{(1+\alpha) / 2}} \E \left[ \left\| \tilde{A}^{\frac{\alpha}{2}}g_s^N \right\|_{\mathbb{L}^p(\T^2)}^q \right]^{\frac 1 q},
\end{aligned}
\]
where $\lesssim$ and $\simeq$ mean that the inequalities and the equalities hold up to a constant.
Therefore, using the above, and the Lemma \ref{lemma} we get:
\[
\E \left[ \left\| \tilde{A}^{\frac{\alpha}{2}} g_t^N \right\|_{\LL^p}^q \right]^{1/q} \leq C + \int_0^t \frac{C_{\nu, T}}{(t-s)^{(1+\alpha) / 2}} \E \left[ \left\| \tilde{A}^{\frac{\alpha}{2}} g_s^N \right\|_{\LL^p}^q \right]^{1/q}ds
\]
and we conclude the proof by means of Gronwall's lemma.
\end{proof}

\begin{prop}
    \label{prop2}
    Let $\gamma \in (0,\frac{1}{2})$ and $q' \geq 2$. Then there exists a positive constant $C$ which does not depend on $N$ such that:
    \[
    \E \bigg[ \int_0^T \int_0^T \frac{ \|  g_t^{N, \pm} - g_s^{N, \pm} \|_{-2,2}^{q'}}{   |t-s|^{1+q'\gamma} } ds dt \bigg] \leq C.
    \]
\end{prop}

\begin{proof}
Again, we omit to specify the sign of $g^N$ and we will work with $g^{N, +}$. The proof for $g^{N, -}$ is identical. Let $q^{\prime} \geqslant 2$. We use the fact that $\mathbb{L}^2\left(\mathbb{T}^2\right) \subset \mathbb{H}_2^{-2}(\T^2)$ with continuous embedding, and that the linear operator $\Delta$ is bounded from $\mathbb{L}^2\left(\mathbb{T}^2\right)$ to $\mathbb{H}_2^{-2}(\T^2)$. See e.g. \cite{belyaev2022multipliers} for these results.

Let us first recall that, from interpolation, by the previous proposition and using the fact that $\left\|g_t^N\right\|_{\mathbb{L}^1\left(\mathbb{T}^2\right)}=1$, we have, for any $\theta \in(0,1)$:
\begin{equation}
\label{eq:interp_g}
\mathbb{E}\left[\left\|g_t^N\right\|_{0,2}^{q^{\prime}}\right] \leqslant \mathbb{E}\left[\left\|g_t^N\right\|_{0, p}^{\theta q^{\prime}}\left\|g_t^N\right\|_{\mathbb{L}^1\left(\T^2\right)}^{(1-\theta) q^{\prime}}\right] \leqslant \mathbb{E}\left[\left\|g_t^N\right\|_{0, p}^{\theta q^{\prime}}\right] \leqslant C_{q^{\prime}, T}
\end{equation}

We then observe that:
\begin{equation}
\begin{aligned}
g_t^N(x) - g_s^N(x) &=  \int_s^t \left\langle \mu_r^N, F \left(K * g^{N, +}_r - K * g^{N, -}_r \right) \cdot \nabla V^N(x-\cdot) \right\rangle dr \\
&+ \frac{\sqrt{2\nu}}{N}\sum\limits_{i=1}^N \int_s^t \nabla V^N(x-X^{i,N}_r) d B^i_r + \int_s^t A g_r^N(x) dr \\
&+ \sum\limits_{k \in \mathbb{N}} \frac{1}{N} \sum\limits_{i=1}^N\int_s^t \sigma_k(X^{i,N}_r) \cdot \nabla V^N(x-X^{i,N}_r) dW^k_r \\
&+ \sum\limits_{k \in \mathbb{N}} \int_s^t \left\langle \mu_r^N, \left( \sigma_k \cdot \nabla \sigma_k \right) \cdot \nabla V^N(x-\cdot) \right\rangle dr
\end{aligned}
\end{equation}

Therefore we have
\[
\begin{aligned}
\mathbb{E}\left[\left\|g_t^N(x)-g_s^N(x)\right\|_{-2,2}^{q^{\prime}}\right] & \leq(t-s)^{q^{\prime}-1} \int_s^t \mathbb{E}\left[\left\|\left\langle \mu_r^N, F \left(K * g^{N, +}_r - K * g^{N, -}_r \right) \nabla V^N(x-\cdot)\right\rangle\right\|_{-2,2}^{q^{\prime}}\right] d r \\
& +(t-s)^{q^{\prime}-1} \frac{1}{2} \int_s^t \mathbb{E}\left[\left\| A g_r^N(x)\right\|_{-2,2}^{q^{\prime}}\right] d r \\
&+ \mathbb{E}\left[\left\|\frac{\sqrt{2 \nu}}{N} \sum_{i=1}^N \int_s^t \nabla\left(V^N\right)\left(x-X_r^{i, N}\right) d B_r^i\right\|_{-2,2}^{q^{\prime}}\right] \\
&+ \sum\limits_{k \in \mathbb{N}} \mathbb{E} \left[ \left\| \frac{1}{N} \sum\limits_{i=1}^N\int_s^t \sigma_k(X^{i,N}_r) \cdot \nabla V^N(x-X^{i,N}_r) dW^k_r \right\|_{-2,2}^{q'} \right] \\
&+ \sum\limits_{k \in \mathbb{N}}\mathbb{E} \left[ \left\| \frac{1}{N}\sum\limits_{i=1}^N \int_s^t \left( \sigma_k(X^{i,N}_r) \cdot \nabla \sigma_k(X^{i,N}_r) \right) \cdot \nabla V^N(x-X^{i,N}_r) dr\right\|_{-2,2}^{q'}\right] 
\end{aligned}
\]

To estimate the first term we observe first that
$$
\begin{aligned}
& \mathbb{E}\left[\left\|\left\langle \mu_r^N, F \left(K * g^{N, +}_r - K * g^{N, -}_r \right) \nabla V^N(x-\cdot)\right\rangle\right\|_{-2,2}^{q^{\prime}}\right] =\mathbb{E}\left[\left\| \nabla \left( \left( \mu_r^N F \left(K * g^{N, +}_r - K * g^{N, -}_r \right) \right) * V^N\right)\right\|_{-2,2}^{q^{\prime}}\right] \\
& \leq \mathbb{E} \left[ \left\| \left(\mu_r^N F \left(K * g^{N, +}_r - K * g^{N, -}_r \right)\right) * V^N \right\|_{-1,2}^{q^{\prime}}\right] \\
& \leq C_M \mathbb{E}\left[\left\|g_t^N\right\|_{\mathbb{L}^2\left(\mathbb{T}^2\right)}^{q^{\prime}}\right] \leq C .
\end{aligned}
$$
For the second term, we write
$$
\mathbb{E}\left[\left\| A g_r^N\right\|_{-2,2}^{q^{\prime}}\right] \leq C \mathbb{E}\left[\left\|g_r^N\right\|_{\mathbb{L}^2\left(\mathbb{T}^2\right)}^{q^{\prime}}\right] \leq C_{q^{\prime}, T}
$$
using \eqref{eq:interp_g}. We bound the stochastic terms in the following way:
\begin{equation}
\begin{aligned}
\mathbb{E}\left[ \bigg\| \frac{1}{N} \sum_{i=1}^N \int_s^t \nabla V^N (x-\right. & \left.\left.X_r^{i, N}\right) d B_r^i \bigg\|_{-2,2}^{q^{\prime}} \right] \\
& \leq C_{q^{\prime}} \mathbb{E}\left[\frac{1}{N^2} \sum_{i=1}^N \int_s^t\left\|\nabla\left(V^N\right)\left(x-X_r^{i, N}\right)\right\|_{-2,2}^2 d r\right]^{q^{\prime} / 2},
\end{aligned}
\end{equation}
and:
\begin{equation}
\begin{aligned}
& \mathbb{E}\left[ \left\| \frac{1}{N} \sum_{i=1}^N \int_s^t \sigma_k(X_r^{i,N}) \cdot \nabla V^N(x-X_r^{i, N}) d W_r^i \right\|_{-2,2}^{q^{\prime}} \right] \\
& \leq C_{q^{\prime}} \left\| \sigma_k \right\|_{\LL^{\infty}} \left\| \nabla \sigma_k \right\|_{\LL^{\infty}} \mathbb{E}\left[\frac{1}{N^2} \sum_{i=1}^N \int_s^t\left\|\nabla\left(V^N\right)\left(x-X_r^{i, N}\right)\right\|_{-2,2}^2 d r\right]^{q^{\prime} / 2}
\end{aligned}
\end{equation}
and we observe that
\begin{equation}
\label{eq:estimate_for_VN}
\begin{aligned}
\frac{1}{N^2} \int_{\mathbb{R}} & \sum_{i=1}^N \int_s^t\left\|(\mathrm{I}-A)^{-1} \nabla\left(V^N\right)\left(x-X_r^{i, N}\right)\right\|_{\LL^2}^2 d r d x \\
& =(t-s) \frac{1}{N}\left\|V^N\right\|_{-1,2}^2 \leq(t-s) \frac{1}{N}\left\|V^N\right\|_{0,2}^2 \leq C N^{2 \beta-1}(t-s) \leq C(t-s) .
\end{aligned}
\end{equation}
Now, it remains to estimate:
\begin{equation}
\begin{aligned}
    & \mathbb{E} \left[ \left\| \frac{1}{N}\sum\limits_{i=1}^N \int_s^t \left( \sigma_k(X^{i,N}_r) \cdot \nabla \sigma_k(X^{i,N}_r) \right) \cdot \nabla V^N(x-X^{i,N}_r) dr\right\|_{-2,2}^{q'}\right] \\
    & = \mathbb{E} \left[ \left( \left\| \frac{1}{N}\sum\limits_{i=1}^N \int_s^t \left( \sigma_k(X^{i,N}_r) \cdot \nabla \sigma_k(X^{i,N}_r) \right) \cdot \nabla V^N(x-X^{i,N}_r) dr\right\|_{-2,2} \right)^{q'}\right] \\
    & = \mathbb{E} \left[ \left( \left\| \frac{1}{N}\sum\limits_{i=1}^N \int_s^t \left( \sigma_k(X^{i,N}_r) \cdot \nabla \sigma_k(X^{i,N}_r) \right) \cdot \nabla V^N(x-X^{i,N}_r) dr\right\|_{-2,2}^2 \right)^{q'/2}\right] \\
    & \overset{\text{Jensen}}{\lesssim} \mathbb{E} \left[ \left( \frac{1}{N}\sum\limits_{i=1}^N \int_s^t \left\| \left( \sigma_k(X^{i,N}_r) \cdot \nabla \sigma_k(X^{i,N}_r) \right) \cdot \nabla V^N(x-X^{i,N}_r)\right\|_{-2,2}^2 dr \right)^{q'/2}\right] \\
    & \leq \left\| \sigma_k \right\|_{\LL^{\infty}}^{q'} \left\| \nabla \sigma_k \right\|_{\LL^{\infty}}^{q'} \mathbb{E} \left[ \left( \frac{1}{N}\sum\limits_{i=1}^N \int_s^t \left\| \nabla V^N(x-X^{i,N}_r)\right\|_{-2,2}^2 dr \right)^{q'/2}\right] \\
\end{aligned}
\end{equation}
We control the last term in the same way we did in \eqref{eq:estimate_for_VN}.
To conclude the lemma, we need to divide all the terms by $|t-s|^{1+q^{\prime} \gamma}$. From the previous estimates, we always get a term of the form $|t-s|^{\varepsilon}$ with $\varepsilon := q' \gamma <1$ (using the assumption $\gamma<\frac{1}{2}$ and $q' \geq 2$).
\end{proof}

%%%%%%%%%%%%%%%%%%%%%%%%%%%%%%%%%%%%%%%%%%%%%%%%%%%%%%%%%%%%%
\newpage

%%%%%%%%%%%%%%%%%%%%%%%%%%%%%%%%%%%%%%%%%%%%%%%%%%%%%
\section{Appendix}

\subsection{Properties of the Biot-Savart kernel}
If \(\omega\) is a scalar field defined on $\T^2$, then the function \(u = K * \omega\) satisfies the equation \(\curl(u) = \omega\), where
\begin{equation}
\label{eq:biot-savart_kernel}
K(x) := \frac{i}{2 \pi} \sum_{k \in \Z^2: k \neq 0} \exp\left(i k \cdot x\right) \frac{k^{\perp}}{k^2}, \quad x \in \T^2.
\end{equation}
Here, the convolution is considered in relation to the group structure of \(\T^2\). It can be shown that \(K \in \LL^1(\T^2)\). In fact, we have \(K = \nabla^{\perp} G\), where
\begin{equation}
\label{eq:green_function}
G(x) := \frac{1}{2 \pi} \sum_{k \in \Z^2; k \neq 0} \frac{\exp\{i k \cdot x\}}{|k|^2}.
\end{equation}
Thus, we can state the following proposition (see the appendix for a proof):
\begin{prop}
The function $G$ defined in \ref{eq:green_function} is in $C^{\infty}\left(\mathbb{T}^2 \backslash\{0\}\right)$. Moreover, its behaviour in 0 is given by
$$
|G(x)| \leq C(-\log |x|+1) \quad x \in \T^2
$$
and that of its derivative $D^{(n)}, n$ positive integer, by
$$
\left|D^n G(x)\right| \leq C_n\left(|x|^{-n}+1\right) .
$$
Therefore we also have the following:
\[
\begin{aligned}
\left|K(x)\right| \leq C(|x|^{-1}+1).
\end{aligned}
\]
where $K$ is the Biot-Savart kernel defined in $\eqref{eq:biot-savart_kernel}$
\end{prop}

\begin{proof}
It is easy to see that the Fourier expansion of $G$ is
$$
G(x)=-\frac{1}{4 \pi^2} \sum_{k \in \mathbb{Z}^2, k \neq 0} \frac{1}{|k|^2} e^{2 \pi i k \cdot x}
$$

Since this expression is not helpful in the analysis of regularity around 0, we will use the solution $v$, in $\LL^2\left([0, T] \times \mathbb{T}^2\right)$, of the heat equation
$$
\partial_t v=\Delta v
$$

with initial condition $v_0=\delta_0-1$ (more precisely, $v_t \rightharpoonup \delta_0-1$ as $t \rightarrow 0$ ). It is easy to see that this unique solution can be expressed in two ways: one with its Fourier expansion, which is
\begin{equation}
\label{eq:heat_equation_1}
    v(t, x)=\sum_{k \in \mathbb{Z}^2, k \neq 0} e^{-4 \pi^2|k|^2 t} e^{2 \pi i k \cdot x},
\end{equation}

To show $v_t \rightharpoonup \delta_0-1$ as $t \rightarrow 0$ it is sufficient to show that:
\[
\left\langle v(t, \cdot), e^{2\pi i k \cdot x} \right\rangle = \begin{cases}
    e^{-4 \pi^2|k|^2 t} \text{ if } k \neq 0, \\
    0 \text{ if } k = 0.
\end{cases} \xrightarrow[t \to 0^+]{} \begin{rcases*}
-1 \text{ if } k \neq 0, \\
    0 \text{ if } k = 0.
\end{rcases*} = \left\langle \delta_0 - 1,  e^{2\pi i k \cdot x} \right\rangle, 
\]
for any $k \in \Z^2$.
The other one is with Gaussian densities, that is:
\begin{equation}
\label{eq:heat_equation_2}
v(t, x)=-1+\frac{1}{4 \pi t} \sum_{l \in \mathbb{Z}^2} \exp \frac{-|x-l|^2}{4 t} .
\end{equation}
To obtain this, we use the following argument. If one works with $\R^d$ instead of $\T^d$, we have that the heat kernel is given as:
\[
\left(\rho_t f\right)(x)=\int_{\mathbb{R}^d} \frac{1}{(4 \pi t)^{d / 2}} e^{-|x-y|^2 /(4 t)} f(y) d y, \quad f \in \mathbb{L}^p\left(\mathbb{R}^d\right),
\]
and therefore the solution to the heat equation will be given by $\rho_t v_0$. To pass to the case of $\T^2$, one can consider the periodic version of the heat kernel:
\[
\left(\tilde{\rho}_t f\right)(x)=\sum\limits_{l \in \Z^2} \left(\rho_t f\right) (x + 2 \pi l), \quad f \in \mathbb{L}^p\left(\mathbb{T}^d\right).
\]
The definition above makes sense because the sum is locally finite (meaning that, locally, the non-zero terms are finite).

One verifies, e.g. using \eqref{eq:heat_equation_1}, that
$$
\begin{aligned}
G(x) & =-\int_0^{+\infty} v(t, x) \mathrm{d} t=-\int_1^{+\infty} =:-G_1(x)-G_2(x) .
\end{aligned}
$$

Now $G_1$ is in $C^{\infty}\left(\mathbb{T}^2\right)$, as one can see from its Fourier expansion, again from \eqref{eq:heat_equation_1}. For $G_2$ we exploit \eqref{eq:heat_equation_2}.

$$
\begin{aligned}
G_2(x)= & \left(-1+\int_0^1 \frac{1}{4 \pi t} \sum_{l \in \mathbb{Z}^2, l \neq 0} \exp \frac{-|x-l|^2}{4 t} \mathrm{~d} t\right) \\
& +\int_0^1 \frac{1}{4 \pi t} \exp \frac{-|x|^2}{4 t}=: G_3(x)+G_4(x),
\end{aligned}
$$
the sum being between functions on $\mathbb{R}^2$ (though $x$ is still in $[-1 / 2,1 / 2]^2$ ). The first term $G_3$ is $C^{\infty}$ on an open neighborhood of $[-1 / 2,1 / 2]^2$ (e.g. $(-3 / 4,3/4) $: indeed, for any $n$ nonnegative integer, we have
$$
\begin{aligned}
& \int_0^1\left|D^{(n)} \frac{1}{4 \pi t} \sum_{l \in \mathbb{Z}^2, l \neq 0} \exp \frac{-|x-l|^2}{4 t}\right| \mathrm{d} t \\
& \lesssim \int_0^1 t^{-(2 n+1)} \sum_{l \neq 0} \exp \frac{-|x-l|^2}{4 t} \mathrm{~d} t \\
& \lesssim \int_0^1 t^{-(2 n+1)} \sum_{h=1}^{\infty} \exp \frac{-h}{c t} \mathrm{~d} t \\
& \sim \int_0^1 t^{-(2 n+1)} e^{-1 /(c t)} \mathrm{d} t<+\infty,
\end{aligned}
$$
for some $c>0$ independent of $x$, when $x$ is in $(-3 / 4,3/4)$. The second term $G_4$ is in $C^{\infty}((-3 / 4,3/4)\backslash \{0\})$. So $G$ is in $C^{\infty}(\mathbb{T}^2 \backslash \{0\})$. For the behaviour at 0, this is given by that of $G_4$ at $0$, which is analysed by standard techniques: we have, with the change of variable $s=|x|^{-1 / 2} t$,
$$
G_4(x) \sim \int_0^{|x|^{-1 / 2}} s^{-1} e^{-1 /(4 s)} \mathrm{d} s \sim-\log |x|
$$
and, for $n \geq 1$,
$$
\left|D^{(n)} G_4(x)\right| \sim|x|^{-n} \int_0^{|x|^{-1 / 2}} s^{-(2 n+1)} e^{-1 /(4 s)} \mathrm{d} s \sim|x|^{-n} .
$$
\end{proof}

More in general we have the following regularizing properties. These follow from one of Green's functions that can be found in \cite{grubb2016regularity}, section 3.

\begin{lemma}
The linear operator $K[\omega]=-\nabla^{\perp}\left(-\Delta \right)^{-1} \omega$, defined first for $\omega \in$ $C^{\infty}(\T^2)$, extends to continuous linear maps (still denoted by $K$ )
$$
\begin{aligned}
& K: \LL^2(\T^2) \rightarrow \LL^q\left(\T^2 ; \mathbb{R}^2\right), q \in[1, \infty), \\
& K: \LL^p(\T^2) \rightarrow C\left(\T^2 ; \mathbb{R}^2\right), p \in(2, \infty), \\
& K: \HH^{\alpha, p}(\T^2) \rightarrow \HH^{\alpha+1, p}\left(\T^2 ; \mathbb{R}^2\right), \quad p \in(1, \infty), \alpha \geq 0 .
\end{aligned}
$$
Moreover, for all the above extensions $K[\omega]$ is divergence-free (in the sense of distributions).
\end{lemma}
\begin{lemma}
We have that:
    \begin{equation}
    \left\| u_t\right\|_{\LL^{\infty}(\T^2)} = \left\| K \star \omega_t \right\|_{\LL^{\infty}(\T^2)} \leq \left\| K \right\|_{\LL^{1}(\T^2)} \left\| \omega_t\right\|_{\LL^{\infty}(\T^2)}
    \end{equation}
    for all $t \in [0,T]$.
\end{lemma}

A proof of this lemma can be found in \cite{Brze_niak_2016} (Lemma 2.17).
We will also need the following result.
\begin{lemma}
    $K$ can be extended to linear and continuous operators:
    \[
    K:=\nabla^{\perp}(-\Delta)^{-1} : H^{-\alpha}_2(\T^2) \to H^{-\alpha+1}_2(\T^2),
    \]
    for every $\alpha \in [0,+\infty]$.
\end{lemma}
\begin{proof}
    As shown in \cite{marchioro_pulvirenti}, $u:=K \star v$ can be expressed in Fourier series as:
    \[
    u(x)=\frac{i}{2 \pi} \sum_{k \in Z^2 ; k \neq 0} \exp \{i k \cdot x\} \hat{v}(k) \frac{k^{\perp}}{|k|^2} .
    \]
    Therefore if $\alpha \geq 0$:
    \[
    \begin{aligned}
    & \| u \|_{\HH^{1-\alpha}_2(\T^2)}^2 := \sum\limits_{k \in \Z^2, k \neq 0} (1+|k|^2)^{1-\alpha}|\hat{v}(k) \frac{k^{\perp}}{|k|^2}|^2 = \sum\limits_{k \in \Z^2, k \neq 0} (1+|k|^2)^{1-\alpha}|\hat{v}(k)|^2 \frac{1}{|k|^2} \\
    & = \sum\limits_{k \in \Z^2, k \neq 0} \frac{(1+|k|^2)}{|k|^2}(1+|k|^2)^{-\alpha}|\hat{v}(k)|^2 \leq 2 \sum\limits_{k \in \Z^2} (1+|k|^2)^{-\alpha}|\hat{v}(k)|^2 = 2 \| v \|_{\HH^{1-\alpha}_2(\T^2)}^2.
    \end{aligned}
    \]
    This completes the proof.
\end{proof}

\subsection{The approximation of the identity $V^N$}

Let us recall that $\T^2=[-\pi,\pi]^2/2\pi\Z^2$.
We show below an explicit example of the sequence $V^N$ as in the assumption \ref{ass:assprinci}. Let us define first:
\begin{equation}
    V(x):=\begin{cases}
        c \cdot e^{\frac{1}{\pi^2-4|x|^2}} \text{ if } |x| \leq \frac{\pi}{2}, \\
        0 \text{ otherwise },
    \end{cases}
\end{equation}
where the constant $c$ is such that $\int_{\T^2} V(y)dy = 1$. Define also $V^N(x):= N^{2\beta}V(N^{\beta}x)$ for $x \in [-\pi, \pi]^2$.
We prove that $V^N$ approximates the identity. First of all:
\begin{equation}
\begin{aligned}
\int_{\mathbb{T}^2} V^N(x-y) d y= 
\int_{[-\pi,\pi]^2} V^N(x-y) d y = \int_{x+ [-\pi,\pi]^2} V^N(y) d y = \int_{\T^2} V^N(y) d y = 1
\end{aligned}
\end{equation}
Next we show that, for $N$ sufficiently large:
\[
\int_{\mathbb{T}^2 \backslash \mathbb{T}_\delta^2(x)}\left|V^N(x-y)\right| d y \leq \epsilon.
\]
Denoting by $B_{\T^2}(x,\delta)$ the class of $\mathbb{T}_\delta^2(x)$ in $[-\pi,\pi]^2$:
\begin{equation}
\begin{aligned}
\int_{\mathbb{T}^2 \backslash \mathbb{T}_\delta^2(x)} V^N(x-y) d y = & \int_{[-\pi,\pi]^2 \backslash B_{\T^2}(x,\delta)} V^N(x-y) d y = \int_{[-\pi,\pi]^2 \backslash B_{\T^2}(x,\delta)} V^N(x-y) d y \\
& = \int_{x+ [-\pi,\pi]^2 \backslash B_{\T^2}(0,\delta)} V^N(y) d y = \int_{[-\pi,\pi]^2 \backslash B_{\T^2}(0,\delta)} V^N(y) dy = \\
& = \int_{\T^2 \backslash B(0,\delta)} V^N(y) dy = \int_{\T^2 \backslash B(0,\delta)} N^{2\beta} \cdot V(N^{\beta}y) dy.
\end{aligned}
\end{equation}
Finally if $\frac{1}{N^{\beta}} \frac{\pi}{2} < \delta$ we have the result because, for $N$ sufficiently large:
\[
\text{supp} (V(N^{\beta} \cdot)) \subset B(0,\delta).
\]

\begin{lemma}
\label{lem:K_negative_epsilon}
If $2/3 < \epsilon < 1$ then:
\begin{equation}
\label{eq:1}
\|K \star v\|_{\HH^{1-\varepsilon/2}_2} \lesssim \|v\|_{\HH^{-\varepsilon/2}_2}.
\end{equation}
\end{lemma}

\begin{proof}
One has:
\[
\|K * v\|_{\HH^{1-\varepsilon/2}_2(\T^2)}=\left\|\nabla^{\perp}(-\Delta)^{-1} v\right\|_{\HH^{1-\varepsilon/2}_2(\T^2)} \lesssim \|v\|_{\HH^{-\varepsilon/2}_2(\T^2)},
\]
because, since $\varepsilon \geq 2/3$, by Sobolev embeddings and $-1+\varepsilon \geq -\varepsilon/2$:
\[
\| v \|_{\HH^{-\varepsilon/2}(\T^2)} \lesssim \| v \|_{\HH^{-1+\varepsilon}(\T^2)}.
\]
\end{proof}

%%%%%%%%%%%%%%%%%%%%%%%%%%%%%%%%%%%%%%%%%%%%%%%%%%%%%%%%%%%%%%%%%%%%%%%%%%%%%%%%%%%%%%%%%
\newpage

%%%%%%%%%%%%%%%%%%%%%%%%%%%%%%%%%%%%%%%%%%%%%%%%%%%%%

%%%%%%%%%%%%%%%%%%%%%%%%%%%%%%%%%%%%%%%%%%%%%%%%%%%%%

%%%%%%%%%%%%%%%%%%%%%%%%%%%%%%%%%%%%%%%%%%%%%%%%%%%%%
\section*{Funding and authorship}

\textbf{Funding:} The first author has been funded by the Department of Mathematics, Imperial College London through a PhD studentship.
\\

\textbf{Authorship:} The first author is the corresponding author. Both authors contributed equally to the formulation of the results, the proofs, and the writing of the manuscript. All authors have read and approved the final version of the paper. The e-mail address of the corresponding author is \href{mailto:f.giovagnini23@imperial.ac.uk}{f.giovagnini23@imperial.ac.uk}.

\bibliographystyle{abbrvnat}
\bibliography{bibliography}

%%%%%%%%%%%%%%%%%%%%%%%%%%%%%%%%%%%%%%%%%%%%%%%%%%%%%
\end{document}